\documentclass[a4paper, 11pt]{article}

\usepackage{calc, amsmath, amsthm, a4, latexsym, amssymb, layout, subfigure, color, bm, xspace}
\usepackage[dvips]{graphicx}

\definecolor{comcolor}{rgb}{0.9,0.3,0.3}
\definecolor{starcolor}{rgb}{0.3,0.3,0.9}
\definecolor{hscolor}{rgb}{0.9,0.6,0.5}

\newtheorem{thm}{Theorem}[section]
\newtheorem{lemma}[thm]{Lemma}
\newtheorem{corollary}[thm]{Corollary}
\newtheorem{prop}[thm]{Proposition}

\newtheorem*{thm14ii}{Theorem \ref{thm:fractional_moments}, part (ii)}

\theoremstyle{definition}

\newtheorem{rem}[thm]{Remark}

\newcommand{\be}[1]{\begin{equation}\label{#1}}
\newcommand{\ee}{\end{equation}}
\newcommand{\ba}{\begin{array}}
\newcommand{\ea}{\end{array}}
\newcommand{\bal}{\begin{aligned}}
\newcommand{\eal}{\end{aligned}}

\newcommand{\ie}{i.e.\@\xspace}
\newcommand{\R}{\mathbb{R}}

\newcommand{\N}{\mathbb{N}}
\newcommand{\Q}{\mathbb{Q}}

\newcommand{\E}{\mathbb{E}}
\newcommand{\p}{\mathbb{P}}

\newcommand{\calF}{\mathcal{F}}
\newcommand{\calG}{\mathcal{G}}

\newcommand{\calV}{\mathcal{V}}
\newcommand{\calW}{\mathcal{W}}

\newcommand{\Var}{\mathrm{Var}}

\newcommand{\1}{1\hspace{-0.098cm}\mathrm{l}}

\newcommand{\la}{\lambda}

\newcommand{\eps}{\varepsilon}
\newcommand{\om}{\omega}

\newcommand{\invis}[1]{}

\newcommand{\ra}{\rightarrow}

\newcommand{\ssup}[1] {{\scriptscriptstyle{({#1}})}}

\begin{document}

\begin{center}
{\LARGE \bf The near-critical scaling window for directed polymers on disordered trees}\\
\vspace{0.7cm}
{\sc\large Tom Alberts\footnote{California Institute of Technology, Department of Mathematics, 253-37, 1200 E California Blvd, 91106, Pasadena, CA, USA.}
and Marcel Ortgiese\footnote{Institut f\"ur Mathematik, MA 7-4, Fakult\"at II,
        Technische Universit\"at Berlin,
        Stra\ss e des 17.~Juni 136, 10623 Berlin, Germany.
        }}\\
\end{center}

\vspace{0.3cm}

\begin{abstract}
We study a  directed polymer model in a random environment on infinite
binary trees. The model is characterized by a phase transition
depending on the inverse temperature.
We concentrate on the asymptotics of the partition function in
the near-critical regime, where the inverse temperature is a small perturbation away from the critical one with the perturbation converging to zero as the system size grows large. Depending on the speed
of convergence we observe very different asymptotic behavior. If the perturbation is
small then we are inside the critical window and observe the same decay of the
partition function as at the critical temperature.
If the perturbation
is slightly larger
the near-critical scaling
leads to a new range of asymptotic behaviors,
which at the extremes
match up with the already known rates for the sub- and super-critical regimes.
We use our results to identify the size of the fluctuations
of the typical energies under the critical Gibbs measure.

  \par\medskip

\end{abstract}


\section{Introduction and main results \label{sec:intro}}

\subsection{Introduction \label{subsec:intro}}

Polymers in a random environment are classical examples
of models driven by an energy - entropy competition. In these models,
the directed polymer corresponds to the path of a random walk on a lattice
while the environment is a field of i.i.d.\! random variables.
The path's interaction with the random environment is governed by an (inverse)
temperature parameter $\beta$. As the temperature
is decreased, the behavior changes from an entropy-dominated
regime with a diffusively behaving polymer, to an energy
dominated regime in which the polymers prefer regions where the environment
is especially favorable. While the large temperature phase is fairly well understood, there are many open problems in the energy dominated regime (especially for general environments).

Beginning with Derrida-Spohn \cite{DS88} it was realized that changing
the underlying space and studying directed polymers on trees allows the use of different techniques. Most notably, one can use
the self-similarity of the graph to exactly compute several quantities. The basic model is the following: let $T$ be an infinite binary tree and to each vertex $v \in T$ attach a random variable $\omega(v)$. The collection $\{ \omega(v) \}_{v \in T}$ is assumed to be i.i.d., and throughout we assume that
\begin{align*}
e^{\lambda(\beta)} := \E \big[ e^{\beta \omega} \big] < \infty \textrm{ for all } \beta \in \R.
\end{align*}
Let $o$ be the root of the tree and $|v|$ denote the generation of each vertex. If $|v| = n$ let $(o = v_0, v_1, \ldots, v_n)$ be the unique path of vertices from $o$ to $v$. Since the path is unique we can refer to each polymer of length $n$ by the last vertex. The interaction with the environment is described by introducing the Gibbs measure $\mu_n^\ssup{\beta}$ which assigns to each polymer $v$ the probability
\be{eqn:polymer_measure}
\mu_n^\ssup{\beta} (v) = \frac{1}{Z_n(\beta)} \exp\Big\{ -\beta H(v) \Big\} ,
\ee
where the energy $H(v)$ is defined by $H(v) = -\sum_{j=1}^n \om(v_j)$,
and the normalizing partition function at level $n$ is given by
\begin{align*}
Z_n(\beta) := \sum_{|v|=n} \exp \Big\{- \beta H(v) \Big\}.
\end{align*}
Note that by interpreting the energies as spatial positions, one
actually obtains a branching random walk (in our case with dyadic branching)
and many results were first described in that language.
Observe that $\E \left[ Z_n(\beta) \right] = e^{n \lambda(\beta) + n \log 2}$, and in fact it is easy to see that $$
W_n(\beta) := Z_n(\beta)/\E Z_n(\beta)$$ is a positive martingale with respect to the filtration $\calW_n := \sigma(\omega(v) : |v| \leq n)$. Applying Kolmogorov's $0$-$1$ law gives the usual dichotomy that exactly one of the events
\begin{align*}
\lim_{n \to \infty} W_n(\beta) > 0, \quad \lim_{n \to \infty} W_n(\beta) = 0
\end{align*}
is of full probability. The $\beta$ for which the limit is positive are said to be in the weak disorder regime; the remaining $\beta$ are said to be in strong disorder. One of the main advantages of the tree is that there is a complete classification of weak and strong disorder: there exists a $\beta_c \geq 0$ such that the range $0 \leq \beta < \beta_c$ is weak disorder and $\beta \geq \beta_c$ is strong disorder. Moreover, $\beta_c$ is the unique non-negative solution to the equation
\begin{align}\label{critical_eqn}
\lambda(\beta_c) + \log 2 = \beta_c \lambda'(\beta_c).
\end{align}
If no solution exists then $\beta_c = \infty$. See \cite{KP76, Big77} for proofs of this fact. We will assume throughout that $\beta_c < \infty$. At the critical inverse temperature $\beta_c$ there is also a drastic change in the behavior of the free energy, as was first proved in a continuous time analogue in \cite{DS88} and later in the tree case in \cite{BPP93}. The result is that
\begin{align}\label{free_energy_result}
\varphi(\beta) := \lim_{n \to \infty} \frac{1}{n} \log Z_n(\beta) = \left\{
                                                                   \begin{array}{ll}
                                                                     \lambda(\beta) + \log 2, & \beta \leq \beta_c, \\
                                                                     \frac{\beta}{\beta_c} \left( \lambda(\beta_c) + \log 2 \right), & \beta > \beta_c.
                                                                   \end{array}
                                                                 \right.
\end{align}
Observe that the free energy varies continuously with $\beta$ but starts growing linearly once $\beta > \beta_c$. Transferring this result to $W_n(\beta)$
combined with the convexity of $\la$ gives that for $\beta > \beta_c$,
$W_n(\beta)$ decays exponentially fast in $n$.
Note that no statement is made about the decay of the martingale in the critical $\beta = \beta_c$ case, and for a long period of time the exact behavior was unknown. This problem was solved in the important work of Hu and Shi \cite{HS09} where, among many other results, they prove that
\begin{align}\label{eqn:decay_critical_martingale}
W_n(\beta_c) = n^{-\frac12 + o(1)} \textrm{ a.s. }
\end{align}
In particular, this implies that even though $\beta = \beta_c$ is in the strong disorder regime, the partition function decays only polynomially fast rather than exponentially as for $\beta > \beta_c$.

\subsection{Main results \label{subsec:results}}

The main goal of this paper is to probe the phase transition at $\beta_c$ and to see, roughly speaking, ``how far'' it extends on either side of the critical temperature.
More precisely, we consider the system at a temperature $\beta_n$
depending on the system size (parametrized by $n$) and
apply a near-critical scaling of $\beta_n \to \beta_c$ as $n \to \infty$.
Our main result determines what types of asymptotics are exhibited for the different choices for scalings of $\beta_n$. This question was inspired by
the recent work~\cite{AKQ12} on the lattice model in $1+1$ dimensions.

To formulate our results we introduce, for a polymer $v$ in the $n$th generation,
the normalized energy at criticality
\begin{align}\label{vdef}
V(v) = \beta_c \left( H(v) + n \la'(\beta_c) \right) = \beta_c H(v) + n (\lambda(\beta_c) + \log 2),
\end{align}
with the last equality coming from equation \eqref{critical_eqn}. Using this notation we have that
\begin{align}\label{wndef}
W_n := W_n(\beta_c) = \sum_{|v|=n} e^{-V(v)}.
\end{align}
For $\delta > 0$ we introduce the perturbed partition functions
\begin{align} \label{eqn:perturbed_part_funcs}
W_n^{+,\delta} = \sum_{|v| = n} e^{-(1+n^{-\delta}) V(v)}, \mbox{ and }
W_n^{-,\delta} = \sum_{|v| = n} e^{-(1-n^{-\delta}) V(v)}.
\end{align}
This perturbation of the energies corresponds to studying the model
near the critical inverse temperature and is more convenient
than taking $\beta_n \ra \beta_c$ directly. The difference amounts
to a deterministic factor which can be calculated explicitly.

The perturbed partition functions~(\ref{eqn:perturbed_part_funcs}) will be our primary objects of study. We generally refer to $W_n^{\pm, \delta}$ as a either positive or negative perturbations, depending on the sign indicated. In our notation large (small) $\delta$ corresponds to small (large) perturbations, and we frequently refer to a perturbation as being large or small. We consider four different types of perturbations (small positive, small negative, large positive, and large negative) and our main results are on the asymptotic behavior of the corresponding partition functions. We show that the separation between small and large perturbations occurs at $\delta = 1/2$. If the perturbation is small, meaning that $\delta \geq 1/2$, then the perturbed partition function decays at the same rate as the unperturbed partition function $W_n$, see \eqref{eqn:decay_critical_martingale}. This is true for both positive and negative perturbations, and the rate of decay does not depend on $\delta$. However, if the perturbation is large, meaning $0 < \delta < 1/2$, then the asymptotics are different in the positive and negative cases, and the asymptotic rate has an explicit dependence on $\delta$.

\begin{thm}\label{thm:main}
\begin{enumerate}
\item
If $\delta \geq 1/2$, then in probability
\begin{align*}
W_n^{\pm, \delta} = n^{-1/2 + o(1)}.
\end{align*}
\item
If $0<\delta < 1/2$, then, almost surely,
\begin{align*}
W_n^{-, \delta} = \exp \left \{ \frac{\beta_c^2 \lambda''(\beta_c)}{2} n^{1-2\delta}(1 + o(1))  \right \}.
\end{align*}
\item
If $0<\delta < 1/2$, then in probability
\begin{align*}
W_n^{+,\delta} = n^{2\delta - \frac{3}{2} + o(1)}.
\end{align*}
\end{enumerate}
\end{thm}

There are two main features of the theorem that we call attention to. First, it clearly shows the existence of a critical scaling window described in terms of the $\delta$ parameter. The critical value of $\delta$, by which we mean the point at which the perturbation switches from being influential to having no influence, is $\delta = 1/2$. The range $\delta \geq 1/2$ is what we call the critical window since the asymptotic behavior is as if the temperature were already at criticality. The range $0 < \delta < 1/2$ is what we call the near-critical window. In the critical window we see Hu-Shi asymptotics, while in the near-critical window we observe new behavior.

This new behavior inside the near-critical window is also of interest, in particular the non-trivial dependence on $\delta$. The exponents $1- 2 \delta$ and $2 \delta - 3/2$ in parts (ii) and (iii), respectively, may appear arbitrary at first but in fact show that there is a ``smooth'' crossover between what is already known for the sub- and super-critical regimes. To describe this crossover we introduce the random variables
\begin{align}\label{eqn:super_partition}
W_{n, \gamma} = \sum_{|v|=n} e^{-\gamma V(v)}.
\end{align}
for $\gamma > 0$. Clearly $W_{n,1} = W_n$. For $\gamma < 1$ the martingale convergence of $W_n(\beta)$ for $\beta < \beta_c$ implies that
\begin{align*}
W_{n, \gamma} \sim W_\infty(\gamma) \exp \{ c(\gamma) n\}
\end{align*}
as $n \to \infty$, for some positive constant $c(\gamma)$ and $W_\infty(\gamma)$
a positive random variable. Hence as $\delta \downarrow 0$ we expect that $W_n^{-,\delta}$ should exhibit linear exponential growth, and the exponent $1 - 2\delta$ confirms this. Similarly, as $\delta \uparrow 1/2$ we should observe a transition from the exponential growth to the Hu-Shi polynomial decay \eqref{eqn:decay_critical_martingale}. Our proofs are not strong enough to capture the transition to the polynomial behavior, but they do show that the exponential growth disappears.

For $\gamma > 1$ it was shown in \cite[Theorem 1.4]{HS09} that
\begin{align*}
W_{n, \gamma} = n^{-\tfrac32 \gamma + o(1)}
\end{align*}
in probability. As $\gamma \downarrow 1$ there is a discontinuity in the decay exponent, with $n^{-3/2}$ appearing instead of the $n^{-1/2}$ in \eqref{eqn:decay_critical_martingale}. Part (iii) of our theorem shows that the discontinuity is bridged by going through the near-critical window, and that there is a linear interpolation between the previously known exponents at the extremes.

This crossover behavior of exponents is not merely coincidental, but reflects a change in the underlying structure of the polymer measures.
In the subcritical case $\beta < \beta_c$, it is known that the polymer measure $\mu_n^\ssup{\beta}$ chooses paths $v$ whose energy $H(v)$ grows like $-\la'(\beta)n$ (up to first order). In particular, in the tree picture it means that exponentially many polymers contribute to the free energy, see for example~\cite{MO08}.

In the supercritical case,~\cite{Ma11} proves that the partition functions $W_{n,\gamma}$
in~(\ref{eqn:super_partition}) with $\gamma > 1$ converges in law if normalized by $n^{- \frac{3}{2}\gamma}$.
In~\cite{BRV12} the limiting law is identified and used to show that the supercritical Gibbs measure converges to a purely atomic measure of Poisson-Dirichlet type.
The convergence of the Gibbs measure for a continuous-time analogue
was already described in~\cite{BK04} for generalized random energy models.
However, more is known about the structure of the Gibbs measure.
As pointed out in~\cite{ABK2} for the case corresponding to
branching Brownian motion,
in the supercritical case the polymer measure is
concentrated on those paths whose energy is within constant
order from the minimal energy.
The latter process of extremal particles was explicitly described in a recent break-through
by~\cite{ABK3} and~\cite{ABBS11} for branching Brownian motion and
after that in~\cite{Ma11} for branching random walks.

In the critical regime,~\cite{JW11} observe that the critical polymer
measure converges, based on the result of~\cite{AidekonShi_ratio} that identifies
the limiting distribution of $n^{\frac{1}{2}}W_n$ as (a constant multiple
of) the limit
of the so-called derivative martingale. However, less is known about the
structure of the Gibbs measure.

Our result about the perturbed partition function also sheds some light
on the critical Gibbs measure.
The fact that perturbations start showing an effect at $\delta = 1/2$ suggests that in the critical window the relevant energies are of order $V(v) \approx n^{1/2}$, and that subexponentially many particles contribute to the partition functions. Using Theorem \ref{thm:main} we easily obtain the following result on the order of the energy at criticality:

\begin{thm}\label{thm:critical_energy} For any $\eps, \eps' > 0$, we have that in probability
\[ \mu_n^\ssup{\beta_c}\Big\{  |v| = n : n^{\frac{1}{2} - \eps} \leq V(v) \leq n^{\frac{1}{2} + \eps'}\Big\}
\ra 1 \, .\]
\end{thm}

\begin{proof}[Proof of Theorem~\ref{thm:critical_energy} assuming Theorem \ref{thm:main}]
Fix $\eps > 0$ and observe that
\[ \bal \mu_n^\ssup{\beta_c}\Big\{ V(v) \leq n^{\frac{1}{2} - \eps}\Big\}
&= \sum_{|v| = n} \frac{e^{-V(v)}}{W_n} \1_{\{ V(v) \leq n^{\frac{1}{2}-\eps} \}} \\
& \leq e^1 \sum_{|v| = n} \frac{e^{-(1+n^{-\frac{1}{2} + \eps}) V(v)}}{W_n}
= e^1 \frac{W_n^{+, \frac{1}{2}-\eps}}{W_n}\, 	.	 \eal\]
By \cite{HS09} we have $W_n = n^{-1/2 + o(1)}$ almost surely, and by Theorem \ref{thm:main} part (iii) we have that
$W_n^{+, \frac{1}{2}-\eps} = n^{-\frac{1}{2}-2\eps+o(1)}$ in probability. Therefore the ratio above converges to zero in probability.

For the remaining bound fix $\eps > 0$ and for $\delta = \frac{1}{2}(1-\eps)$ consider
\[ \bal \mu_n^\ssup{\beta_c}\Big\{ V(v) &\geq n^{\frac{1}{2} + \eps}\Big\}
 = \sum_{|v| = n} \frac{e^{-V(v)}}{W_n} \1_{\{ V(v) \geq n^{\frac{1}{2}+\eps} \}} \\
 & \leq e^{-n^{-\delta}n^{\frac{1}{2}+\eps}} \sum_{|v| = n} \frac{e^{-(1-n^{-\delta})V(v)} }{W_n}
 =  e^{-n^{-\delta}n^{\frac{1}{2}+\eps}} \frac{W_n^{-,\delta}}{W_n} \, .\eal\]
Again, $W_n = n^{-\frac{1}{2}+o(1)}$ almost surely and by Theorem~\ref{thm:main} part (ii) we have that
$$W_n^{-,\delta} \leq \exp \left \{ \frac{\beta_c^2 \lambda''(\beta_c)}{2} n^{1-2\delta}(1+o(1)) \right \}$$ almost surely.
Thus by our choice of $\delta$ the previous expression converges to zero in probability.
 \end{proof}

The proofs also show that the typical behavior of a polymer is that the energy along
its paths $(V(\xi_i))_{i=1}^n$ perform a random walk which stays positive. In the case
of a large positive perturbation with $\delta < \frac{1}{2}$, we have to add the
additional requirement that at the end $V(\xi_n)$ gets pushed down to an unusually
low $n^{\delta}$. In fact, this extends the intuition behind the proofs
of~\cite{HS09} that the main contributing random walk in the supercritical case remains positive, but then has to take an unusually low value at the end.

To prove Theorem \ref{thm:main} we employ the standard technique of deriving the asymptotics of the partition functions from the asymptotics of its fractional moments.  This is the strategy used in \cite{HS09}, and in our situation it is akin to computing the following asymptotics for the fractional moments of the perturbed partition functions:

\begin{thm}\label{thm:fractional_moments}
Let $\gamma \in (0,1)$. Then
\begin{enumerate}
\item for $\delta \geq 1/2$ we have $\E \left[ \left( W_n^{\pm, \delta} \right)^{\gamma} \right] = n^{-\gamma/2 + o(1)}$,
\item for $0 < \delta < 1/2$ we have $\E \left[ (W_n^{-, \delta})^{\gamma} \right] = \exp \left \{ \frac{\gamma}{2} n^{1-2\delta}\beta_c^2 \lambda''(\beta_c) (1 + o(1)) \right \}$
\item for $0 < \delta < 1/2$ we have $\E \left[ \left( W_n^{+, \delta} \right)^{\gamma} \right] = n^{(2 \delta - \frac32)\gamma + o(1)}$.
\end{enumerate}
\end{thm}

In Appendix \ref{sec:moments2asymp} we employ standard arguments to show that Theorem \ref{thm:main} is a corollary of Theorem \ref{thm:fractional_moments}, so
the main focus of this paper is proving Theorem \ref{thm:fractional_moments}.

\subsection{Organization and idea of the proofs \label{subsec:strategy}}

We give here a brief outline of our methods for proving Theorems \ref{thm:main} and \ref{thm:fractional_moments}.
Before we concentrate on our proofs, we will comment on which parts
of the asymptotics can be easily deduced from known results about the
minimal energy (\ie the minimal position of a branching random walk).
We first recall that it was shown in \cite[Thm 1.2]{HS09} that
\begin{align}\label{eq:as_bounds}
\bal \limsup_{n\ra\infty} \frac{1}{\log n} \inf_{|v| = n} V(v) = \frac{3}{2} \,, \quad \liminf_{n\ra\infty} \frac{1}{\log n} \inf_{|v| = n} V(v) = \frac{1}{2}, \, \eal
\end{align}
both almost surely.

\begin{corollary} \label{cor:easy_bounds}
For any negative perturbation, \ie any $\delta > 0$,
\[ W_n^{-,\delta} \geq n^{-\frac{1}{2} + o(1)}, \quad \mbox{almost surely}, \]
for any positive perturbation
\[ W_n^{+,\delta} \leq n^{-\frac{1}{2} + o(1)}, \quad \mbox{almost surely}, \]
and for any perturbation
\[ \limsup_{n \ra \infty} \frac{ \log W_n^{\pm, \delta} }{\log n} \geq - \frac{1}{2},
 \quad \mbox{almost surely.}
\]
\end{corollary}

\begin{rem}\label{rem:trivial_bounds}
These bounds immediately prove that the lower bound for
negative perturbations and the upper bound
for positive perturbations in part (i) of Theorem~\ref{thm:main}
hold (and even in an almost sure sense). In fact our proofs will show that all lower bounds in Theorem \ref{thm:main} hold almost surely. 

Combining the second and third statement of Corollary \ref{cor:easy_bounds} we also see that
$$\limsup_{n \ra\infty} \frac{\log W_n^{+,\delta}}{\log n } = - \frac{1}{2}, \quad \textrm{almost surely}.$$
For $0 < \delta < 1/2$ we also have for the lim inf that
$$\liminf_{n \ra \infty} \frac{\log W_n^{+,\delta}}{\log n} = 2 \delta - \frac{3}{2}, \quad \textrm{almost surely}.$$
\end{rem}

\begin{rem}\label{rem:large_neg_pert}
For negative perturbations the first statement of the corollary completes the proof of the lower bound in part (i) of Theorem \ref{thm:main}. Using this we do not need to prove the lower bound for the fractional moment of $W_n^{-,\delta}$ in the $\delta \geq 1/2$ case (i.e. part (i) of Theorem \ref{thm:fractional_moments}). However, we point out that the fractional moment is an easy corollary of the fractional moments of $W_n$ \cite[Thm. 1.5]{HS09} and the asymptotics of $\inf_{|v|=n} V(v)$.
\end{rem}

\begin{proof}[Proof of Corollary~\ref{cor:easy_bounds}]
For any negative perturbation
we have the lower bound
\[ W_n^{-,\delta} = \sum_{|v| = n} e^{- (1-n^{-\delta})V(v)} \geq e^{n^{-\delta} \inf_{|u| = n} V(u)} \sum_{|v| = n} e^{- V(v)}
\geq e^{\frac{1}{2}n^{-\delta} \log n (1+o(1))} W_n  \, , \]
Since \cite{HS09} implies $W_n = n^{-\frac{1}{2} + o(1)}$ almost surely, we immediately
obtain that
$W_n^{-, \delta} \geq n^{-1/2 + o(1)}$ almost surely.
Using the same idea we also obtain an upper bound for any positive
perturbation, namely
\[ W_n^{+, \delta} = \sum_{|v| = n} e^{- (1+n^{-\delta})V(v)} \leq e^{-\frac{1}{2}n^{-\delta} \log n(1+o(1))} W_n = n^{-1/2 + o(1)}, \]\\[-3mm]
where the last equality is again a consequence of the Hu-Shi asymptotics \eqref{eqn:decay_critical_martingale} for $W_n$.

Finally, we can always obtain a lower bound by only keeping the
minimizing particle in the sum defining the partition function, so that
\[ W_n^{\pm,\delta}  \geq e^{-(1-n^{-\delta}) \inf_{|v|  = n} V(v)} . \]
Now, the $\liminf$ asymptotics~(\ref{eq:as_bounds}) of $\inf_{|v| = n} V(v)$
yield the lower bound on the $\limsup$ asymptotics for $W_n^{\pm, \delta}$.
\end{proof}

The rest of the paper is focused on proving Theorem \ref{thm:fractional_moments}. In several papers on branching processes the \textit{spine method} is the main technique used to understand asymptotics of the process.
The first step is to enlarge the probability space by identifying a special ray, the ``spine'', in the
tree. The second step involves constructing a size-biased probability measure that is tilted towards environments and rays for which the normalized energy $\{V(\xi_i)\}_{i=1}^n$ is typical along the chosen ray $\xi$. Precise definitions and properties of the construction are reviewed in Section \ref{sec:spine_technique}.

The main purpose of this construction is that one can deduce
the asymptotics of the partition function from the
behavior of the normalized energies on the spine $\{ V(\xi_i)\}_{i=1}^n$.
Moreover,
under this tilted measure these normalized energies are in distribution
equal to
to a mean zero random walk. The problem is thus broken into two smaller pieces: first showing that the fractional moments can be estimated by some functional of a simple random walk, and then using random walk methods to estimate the functional.

We explain this strategy in more detail in  the case of small and large positive perturbations.
Our aim is to show that, in a rough sense,  the perturbed partition function
$W_n^{\cdot, \delta}$
decays like the inverse of
\[ g(n) = \left\{ \ba{ll} n^{1/2} & \mbox{if }\delta \geq \frac{1}{2}, \mbox{ any perturbation}, \\
 n^{\frac{3}{2} - 2\delta} & \mbox{if } \delta \in (0,\tfrac{1}{2}), \mbox{ positive perturbation}. \ea \right.
\]
Following the philosophy of the spine method,
we can reduce  a fractional moment
into a functional of a random walk and we eventually show that
for $s \in (0,1)$,
\begin{equation}\label{frac_to_RW} \E[ (g(n) W_n^{\pm, \delta})^{1-s} ] \approx \E [ (g(n) \star (S_n^+)^\alpha)e^{\mp n^{-\delta}S_n} \1_{\{\min_{j} S_j \geq 0 \}} ] , \end{equation}
where $S_n$ is a mean zero random walk with exponential moments, $\star$ is
maximum $\vee$ or
minimum $\wedge$  (depending on whether we consider
an upper or lower bound) and $\alpha > 0$
is a free parameter. If our choice of parameters is correct,
then the right hand side should be essentially constant
(and the dependency on $s$ is hidden in constants).

At this point we can fully notice the effect of the perturbation. If $\delta \geq
\frac{1}{2}$, \ie if the perturbation is small, the
term $e^{-n^{-\delta} S_n}$ is negligible. Hence, the dominating behavior
is that of a random walk conditioned to be positive so that the end point
fluctuates around $n^{\frac{1}{2}}$. However, if we are in the case of
a positive large perturbation the $e^{-n^{-\delta} S_n}$ factor starts to
push the random walk down at the end, so that the dominating contributions
come from random walks that stay positive but end up at a scale $n^{\delta}$
at time~$n$.
In particular, we see that if we choose the parameter $\alpha$
as
\[ \alpha := \alpha(\delta) = \left\{ \ba{ll} 1 & \mbox{if } \delta \geq \frac{1}{2}, \mbox{ any perturbation}, \\
 \frac{3}{2\delta} - 2 & \mbox{if } \delta \in (0,\frac{1}{2}), \mbox{ positive perturbation}.\ea \right.
\]
then, under the dominating behavior in~(\ref{frac_to_RW}),  the random walk satisfies
$(S_n^+)^\alpha \approx g(n)$.

We emphasize that the strategy behind our proofs is highly motivated by
the use of fractional moments and the spine methods in~\cite{HS09}.
However, their proofs cannot be translated directly to deal
with a perturbation of the partition function.
Moreover, in order to be able to concentrate on the new difficulties,
we focus exclusively on the case of a binary tree instead
of general Galton-Watson trees. The binary tree model also appears naturally
as a toy model for polymers.

The \emph{organization} of our paper is as follows:
in Section \ref{sec:spine_technique} we give a brief review of the spine method.
In Section~\ref{sec:large_negative_perturbation}, we deal with the simplest case of a fractional moment bound for a large negative perturbation, which is part (ii) of Theorem \ref{thm:fractional_moments}. Since we only show less refined asymptotics,
we can use simpler methods.
In the remainder of the paper, we carry out the above strategy for
all small and large negative perturbations.
In Section~\ref{section:upper_bound}, we show that we obtain an upper bound
on the fractional moments in terms of a random walk as in~(\ref{frac_to_RW}),
while in Section~\ref{section:lower_bound} we show the corresponding
lower bound.
To complete the proof of the fractional moment estimates, Theorem~\ref{thm:fractional_moments},
we analyze in Section \ref{section:random_walk_expression} the random
walk functional on the right hand side of~(\ref{frac_to_RW}) using a coupling argument
with a Brownian motion.
Appendix~\ref{sec:moments2asymp} shows
how to deduce Theorem~\ref{thm:main} from Theorem~\ref{thm:fractional_moments}.

{\bf Notation:} Throughout the paper, we will use generic constants $c, C > 0$,
whose values may change from line to line. If it is essential, we will indicate
their dependence on parameters.

{\bf Acknowledgments:} We thank the organizers of the $2010$ PIMS Summer School in Probability, where this project originated, and the Fields Institute for hosting us while most of this work was completed. We also thank the organizers of the $2011$ Fields Thematic Program on Dynamics and Transport in Disordered Systems for the invitation to the program.

\section{Spine Method \label{sec:spine_technique}}

\newcommand{\st}{\textbf{SpinedTrees}}

Recall the weight function $V : T \to \R$ defined by \eqref{vdef} and the expression \eqref{wndef} for $W_n$. Let $\st = \{ (\calV, \xi) : \calV = (V(v) : v \in T), \xi \in \partial T \}$ be the space of weights on the vertices of $T$ with a marked spine $\xi$. Let $\calF_n = \sigma(V(v), |v| \leq n; \xi_i, i \leq n)$ be the filtration giving all the information on the weights and spine up to level $n$, and recall that $\calW_n = \sigma(\omega(v) : |v| \leq n)$. Let $\p$ be the probability measure on $\st$ such that the $V(v)$ variables have the distribution defined by \eqref{vdef} with all of the $\omega$ being i.i.d. and $\xi$ chosen uniformly from $\partial T$. Let $\Q$ be the probability measure on $\st$ defined by
\begin{align}\label{Qdef}
\left. \frac{d \Q}{d \p} \right|_{\calF_n} \!\!\!\! (\calV, \xi) = e^{-V(\xi_n) + n \log 2}.
\end{align}
It is easy to check that the latter expression is an $\calF_n$-martingale under $\p$, and hence $\Q$ extends to a measure on all of $\st$. A straightforward computation shows that conditional on the weights $\calV$ (i.e. on $\calW_{\infty}$), the distribution of $\xi_n$ is given by
\begin{align}\label{Qn_dist}
\Q \left( \left. \xi_n = v \right| \calW_{\infty} \right) = \frac{e^{-V(v)}}{W_n}.
\end{align}
Comparing \eqref{Qdef} and \eqref{Qn_dist} with \eqref{vdef} we see that the measure $\Q$ is tilted towards elements of $\st$, for
which the Gibbs measure is large.
Note also that $\Q$ restricted to $\calW_n$ has Radon-Nikod\'ym derivative $W_n$.
Moreover, under $\Q$ the sequence $V(\xi_n)$ turns out to be a random walk with mean zero increments. This is proved in a number of different sources (see \cite{MO08,HS09}, for example) but we recall the basic facts here. For each $n \geq 1$ let $b_n$ be the sibling vertex of $\xi_n$. Define the $\sigma$-algebras $\calG_n, \calG_n^*$ by
\begin{align*}
\calG_n := \sigma(V(\xi_i), \xi_i; i \leq n) \quad \mbox{and} \quad\calG^*_n := \sigma(V(\xi_i), V(b_i), \xi_i; i \leq n).
\end{align*}
Further, let $(S_n, n \geq 0)$ be a random walk with $S_0 = 0$ whose independent increments have the $\Q$-distribution of $V(\xi_1)$. Then there is the following well-known set of results:

\begin{prop}\label{spine_proposition}
Under the measure $\Q$,
\begin{enumerate}
\item the process $(V(\xi_n))_{n \geq 0}$ has the same distribution as the random walk $(S_n)_{n \geq 0}$,
\item for any measurable function $F: \R \to \R$
\begin{align*}
\E_{\Q} \left[ F(S_1) \right] = 2\E \left[ F(V_0) e^{-V_0} \right]
\end{align*}
where $-V_0 = \beta_c \omega - \lambda(\beta_c) - \log 2$,
\item the random variables $(V(\xi_n) - V(\xi_{n-1}), V(b_n) - V(\xi_{n-1}))$ are i.i.d.\ and distributed as $(S_1, V_0)$,
\item conditionally on $\calG_n^*$ the weights $V(v) - V(b_k)$ on the subtree $T(b_k)$ rooted at $b_k$ are independent
of $V(b_k)$ (and independent for each subtree)
and have the same distribution as under the original measure $\p$.
\end{enumerate}
\end{prop}

Choosing $F(x) = x$ in (ii) and using the relation \eqref{critical_eqn} gives that $\E_{\Q}[S_1] = 0$. Hence $S_n$ is a mean zero random walk by parts (i) and (iii).

\section{Large negative perturbations \label{sec:large_negative_perturbation}}

Using the spine method we prove part (ii) of Theorem \ref{thm:fractional_moments}, which is the fractional moments for a large negative perturbation. Combined with the results of Appendix~\ref{sec:moments2asymp} this completes the proof of part (ii) of Theorem \ref{thm:main}.

\begin{thm14ii} For any $\delta \in (0, \frac{1}{2})$ and $s \in (0,1)$, we have that
 \[ \E [ (W_n^{-,\delta})^{1-s} ] = \exp \{\tfrac{1}{2}(1-s)n^{1-2\delta}\beta_c^2 \la''(\beta_c)(1+o(1)) \}
. \]
\end{thm14ii}

\begin{proof}
We first record a standard computation, where we recall the definition of $V$ in~(\ref{vdef})
and compute for any
$0 \leq k \leq n$,
\begin{equation}\label{eqn:moment_calc} \bal \E\Big[ \sum_{|v| = k} e^{- (1-n^{-\delta}) V(v) } \Big]
& = \sum_{|v| = k} \E[ e^{ (1- n^{-\delta}) (\beta_c \sum_{j=1}^k \om(v_j) - k(\la(\beta_c) + \log 2)) } ] \\
& = 2^k \,\E[ e^{(1-n^{-\delta})\beta_c \om} ]^k e^{ - k (1-n^{-\delta}) (\la(\beta_c) + \log 2)} \\
& = \exp \big\{ k \big( \la( (1-n^{-\delta}) \beta_c) -  \la(\beta_c) + n^{-\delta} \beta_c
\la'(\beta_c)  \big) \big\}\\
& = \exp \{ k ( \tfrac{1}{2}  n^{-2\delta}\beta_c^2 \la''(\beta_c) + O (n^{-3\delta}) ) \},
\eal
\end{equation}
where in the penultimate step we used the definition
of $\beta_c$ in~(\ref{critical_eqn}) and a Taylor expansion.

In particular, taking $k=n$ we immediately obtain the \emph{upper bound} on the
fractional moments
by using Jensen's inequality to estimate that for any $s \in (0,1)$,
\[ \E \big[(W^{-,\delta}_n)^{1-s}\big] \leq \E[W^{-,\delta}_n]^{1-s} = \exp \{\tfrac{1}{2}(1-s)n^{1-2\delta}\beta_c^2 \la''(\beta_c) + O(n^{1-3\delta})\} \, , \]
the last equality following from the calculation in~(\ref{eqn:moment_calc}).

We now prove the \emph{lower bound}. Fix $s \in (0,1)$ and observe that
with the notation for the spine technique as introduced
in Section~\ref{sec:spine_technique},
\begin{align*}
\frac{W_n^{-, \delta}}{W_n} = \sum_{|v|=n} e^{n^{-\delta} V(v)} \Q \left( \xi_n = v | \calW_{\infty} \right) = \E_{\Q}[e^{n^{-\delta} V(\xi_n)} | \calW_{\infty}].
\end{align*}
Then the fractional moment can be written as
\[ \E [ (W_n^{-, \delta})^{1-s}] = \E_\Q \Big[  \frac{W_n^{-, \delta}}{W_n}  (W_n^{-,\delta})^{-s} \Big]
= \E_\Q \big[ e^{n^{-\delta} V(\xi_n)}   (W_n^{-,\delta})^{-s}\big],
\]
By conditioning on the weights on the spine $\calG_n$
and applying Jensen's inequality we obtain a lower bound of
\begin{equation}\label{eq:1704-2}\bal  \E [ (W_n^{-, \delta})^{1-s}] & =
\E_\Q \big[ e^{n^{-\delta} V(\xi_n)}   \E [ (W_n^{-,\delta})^{-s}|\calG_n] \big]
 \\ &
\geq \E_\Q \big[ e^{n^{-\delta} V(\xi_n)}   \E [ (W_n^{-,\delta})^{s}|\calG_n]^{-1} \big]
. \eal \end{equation}
We now decompose the tree along its spine to write
\begin{align*}
W_n^{-,\delta} & = e^{-(1-n^{-\delta})V(\xi_n)} \\
& \quad\quad + \sum_{i=1}^n e^{-(1-n^{-\delta})V(\xi_{i-1})} e^{-(1-n^{-\delta})(V(b_i)-V(\xi_{i-1}))} \!\!\!\!\!\!\! \sum_{v \in T_{n-i}(b_i)} e^{-(1-n^{-\delta})(V(v) - V(b_i))},
\end{align*}
where we recall that $b_i$ denotes the sibling of $\xi_i$ in the tree.
Using Proposition \ref{spine_proposition} and the subadditivity inequality
$(\sum_i a_i)^s \leq \sum a_i^s$ for $a_i \geq 0$, we can show that
\begin{align}\label{eq:1704-1} \E_{\Q}[ W_n^{-,\delta} | \calG_n ]^{s} & \leq e^{-s (1-n^{-\delta})V(\xi_n)} \notag\\
& \quad\quad + \sum_{i=1}^n e^{-s (1-n^{-\delta})V(\xi_{i-1})} \E[e^{-(1-n^{-\delta})V_0}]^s \E\Big[ \sum_{|v| = n-j} e^{-(1-n^{-\delta})V(v)} \Big]^s
 \notag\\
& \leq e^{ \frac{1}{2} s n^{1-2\delta} \beta_c^2 \la''(\beta_c) + O(n^{1-3\delta})}
\sum_{i=0}^n e^{-s (1-n^{-\delta}) V(\xi_i)} . \end{align}
The last inequality uses that $ \E[e^{-V_0}]^s = 2^{-s}$ (see Proposition \ref{spine_proposition}, part (ii), for the definition of $V_0$)
and finally the calculation in~(\ref{eqn:moment_calc}).
Combining these last two estimates~(\ref{eq:1704-2}) and~(\ref{eq:1704-1}), we conclude that
\begin{equation}\label{eq:2304-1} \E [ (W_n^{-,\delta})^{1-s}] \geq e^{-\frac{1}{2} s n^{1-2\delta} \beta_c^2 \la''(\beta_c) + O(n^{1-3\delta})} \E_\Q \Big[ \frac{e^{n^{-\delta}V(\xi_n)}}{\sum_{i=0}^n e^{-s (1-n^{-\delta})V(\xi_i)}} \Big]
. \end{equation}
Denoting by $(S_i)_{i \geq 0}$ the random walk introduced in Proposition~\ref{spine_proposition},
we can rewrite the expectation on the right hand side as
\[ \bal
 \E_\Q \Big[ \frac{e^{n^{-\delta}V(\xi_n)}}{\sum_{i=0}^n e^{-s (1-n^{-\delta})V(\xi_i)}} \Big]
& = \E_\Q \Big[ \frac{e^{n^{-\delta}S_n}}{\sum_{i=0}^n e^{-s (1-n^{-\delta})S_i }} \Big]
\geq \frac{1}{n+1}\E_\Q \big[ e^{n^{-\delta}S_n} \1_{\{ \min_{i=1}^n S_i \geq 0 \}}  \big] \\
& \geq \frac{1}{n+1} \E_\Q \big[ e^{n^{-\delta}S_n} \big] \Q\big\{  \min_{i=1, \ldots,n} S_i \geq 0 \big\},
\eal \]
where we used the FKG inequality
noticing that
$(x_i)_{i=1}^n \mapsto \1\{ \min_{i=1,\ldots,n} \sum_{j=1}^i x_i \geq 0 \}$ and $
(x_i)_{i =1}^n \mapsto e^{n^{-\delta} \sum_{j=1}^n x_j}$ are both
increasing functions. For more details of the FKG inequality in a similar
context see e.g.~\cite[Section 2.2]{AD11}.
To complete the proof,  we note that by~(\ref{eqn:moment_calc})
we can calculate the first moment as
$$\bal \E_\Q \big[ e^{n^{-\delta}S_n} \big]
& = \E_\Q \big[e^{n^{-\delta}V(\xi_n)} \big]
= \E_\Q \big[ \sum_{|v| = n} e^{n^{-\delta}V(v)} \Q \{ \xi_n = v | \calW_n \} \big] \\
& = \E [ W_n^{-,\delta}] =   \exp \left \{ \tfrac{1}{2}\beta_c^2 \la''(\beta_c)  n^{1-2\delta} + O(n^{1-3\delta}) \right\},
\eal $$
and for the second term we have that $\Q\{\min_{i=1,\ldots,n} S_i \geq 0\} = n^{-\frac{1}{2} + o(1)}$ by a standard random walk computation. Hence the
latter is negligible compared to the first term and from~(\ref{eq:2304-1})
we can deduce the required lower bound.
\end{proof}

\section{Upper bounds \label{section:upper_bound}}

In this section we find an upper bound on the fractional moments for all positive perturbations and small negative perturbations. The method we use works for all three types of perturbations simultaneously. To unify the argument we write $Y_n = W_n^{\cdot, \delta}$, where $\cdot$ is either $+$ or $-$ depending on whether we are considering a positive or negative perturbation. Define the growth function
\begin{align}\label{eqn:growth_func}
 g(n) = \left\{ \ba{ll}
 n^{1/2} & \mbox{if } Y_n = W_n^{\pm, \delta}, \delta \geq \frac{1}{2},  \\
n^{\frac{3}{2} - 2 \delta} & \mbox{if } Y_n = W_n^{+,\delta}, \delta \in (0,\frac{1}{2}) , \ea \right.
\end{align}
and also let $\gamma_n = \pm n^{-\delta}$ depending on which perturbation is under consideration.

We start by defining the auxiliary quantity $\overline Y_n$, which gives an upper bound as follows:
\[ g(n)Y_n \leq \sum_{|v|=n} (g(n) \vee V^+(v)^\alpha)e^{-(1+ \gamma_n) V(v)} =: \overline Y_n . \]
Here $\alpha = \alpha(\delta)$ is chosen as
\[ \alpha(\delta) = \left\{ \ba{ll} 1 & \mbox{if } Y_n = W_n^{\pm, \delta}, \delta \geq \frac{1}{2}, \\
   \frac{3}{2\delta} - 2 & \mbox{if } Y_n = W_n^{+,\delta}, \delta \in (0,\frac{1}{2}) . \ea \right. \]
This reasoning behind this particular choice of $\alpha$ is discussed
in Section~\ref{subsec:strategy}.

\begin{prop}\label{lemma:upper_reduction_to_random_walk} For all $s \in (0,1)$ there exists a constant $\kappa_0 > 0$ such that for all $\kappa \geq \kappa_0$ there is a $C >0$ satisfying
\[ \E[ \overline{Y}_n^{1-s} ] \leq C( 1 +  \E_\Q[ (g(n) \vee S_n^{\alpha} ) e^{-\gamma_n S_n}\1_A]) + o(1)  \]
as $n \to \infty$, where $A$ is the event
\[ A = \left \{ \min_{0\leq j \leq n} S_j \geq - \kappa \log n , S_n \geq 0 \right \} \, .\]
\end{prop}

\begin{proof}
First note that it is sufficient to prove the proposition for $s$ small, since if it holds for small $s$ then it also holds for all larger $s \in (0,1)$. Indeed, by H\"older's inequality we have that for $s' > s$
\begin{align*}
\E[ \overline{Y}_n^{1-s'} ] \leq \E [ \overline{Y}_n^{1-s} ]^{\frac{1-s'}{1-s}} \leq 1 + \E [ \overline{Y}_n^{1-s} ],
\end{align*}
for $n$ sufficiently large, where we used that $\frac{1-s'}{1-s} < 1$.

Now observe that $\overline{Y}_n$ can be rewritten as
\begin{align*}
\overline{Y}_n = W_n \E_{\Q} \left[ (g(n) \vee V^+(\xi_n)^{\alpha}) e^{-\gamma_n V(\xi_n)} \big| \calW_{\infty} \big. \right],
\end{align*}
and then using the spine techniques of Section~\ref{sec:spine_technique} we obtain that
\be{upper_many_to_one}
\bal \E\big[ \overline Y_n^{1-s}\big] &= \E_{\Q} \big[ \overline Y_n^{\, -s} (g(n) \vee V^+(\xi_n)^\alpha) e^{-\gamma_nV(\xi_n)}\big]. \eal
\ee
As in the proof by \cite{HS09}, the main idea is to
show that the relevant contributions to $\overline Y_n$ only come from the
spine particle $\xi_n$.

We first notice that we can concentrate on the event $\overline Y_n \geq 1$ (on the complement
$\overline Y_n^{1-s}$ is bounded by $1$).
Now define $\underline{V}( \xi_n) = \inf_{i=1,\ldots, n} V(\xi_i)$. Fix $\kappa > 0$ and let $$E := \{ (\calV, \xi) \in \st : \underline{V}(\xi_n) \geq - \kappa \log n , V(\xi_n) \geq 0 \},$$
and notice that we can write $E^c = F_1 \cup F_2$ where
\[ F_1 = \{ (\calV, \xi) : \underline{V}(\xi_n) < - \kappa \log n\} \quad \mbox{and}\quad
F_2 = \{ (\calV, \xi) : V(\xi_n) < 0 , \underline{V}(\xi_n) \geq - \kappa \log n \} \, .\]
We will show that $\E [\overline{Y}_n^{1-s} \1_{\overline Y_n \geq 1} \1_{F_i}] \to 0$ as $n \to \infty$, for $i = 1,2$, so that by equation \eqref{upper_many_to_one} we will have
\[\bal \E\big[ \overline Y_n^{1-s}\big] & \leq 1 + \E [ \overline Y_n^{1-s} \1_{\overline Y_n \geq 1}]
 \\& \leq 1 + \E_{\Q}\big[ \overline Y_n^{\, -s} (g(n) \vee V^+(\xi_n)^\alpha) e^{-\gamma_n V(\xi_n)} \1_{\overline Y_n \geq 1} \1_{E} \big] + o(1) . \eal
\]
This will prove the lemma once we recall that, by Proposition \ref{spine_proposition}, $(S_i)_{i=1}^n$ is a random walk which has the same $\Q$-distribution as the weights $(V(\xi_i))_{i=1}^n$ along the spine.

{\bf Step 1.} We will show that
\[ \E \left[ \overline{Y}_n^{1-s} \1_{F_1} \right] = \E_{\Q} \left[ \overline{Y}_n^{\, -s} (g(n) \vee V^+(\xi_n)^\alpha) e^{-\gamma_n V(\xi_n)} \1_{F_1} \right] \to 0 \]
as $n\ra\infty$. Let $\underline \xi_n$ be the last element of $\xi_1, \xi_2, \ldots, \xi_n$ such that $\underline V(\xi_n) = V(\underline \xi_n)$. Let $b$ be the sibling of $\underline \xi_n$. On $F_1$, we can estimate $\overline Y_n$ from below by
\[ \overline Y_n \geq g(n) e^{- (1+\gamma_n) V(b)} e^{-(1+\gamma_n) \inf_{v \in T_{n - |b|}(b)} V(v)-V(b)} \, .\]
Thus,
\[ \bal &\E_\Q\big[ \overline Y_n^{\, -s}  (g(n) \vee V^+(\xi_n)^\alpha) e^{-\gamma_n V(\xi_n)}\1_{F_1} \big]  \\
& \leq
\E_\Q \big[ (g(n) \vee V^+(\xi_n)^\alpha) g(n)^{-s} e^{s(1+\gamma_n) V(b) - \gamma_n V(\xi_n) } e^{s(1+\gamma_n)
\inf_{v \in T_{n - |b|}(b)} V(v)-V(b)}\1_{F_1} \big] \\
& \leq
C \E_\Q \big[ (g(n) \vee V^+(\xi_n)^\alpha) g(n)^{-s} e^{s(1+\gamma_n) V(\underline \xi_n) - \gamma_n V(\xi_n) } \\
& \hspace{6cm} \times
\sup_{k = 0,\ldots, n}\E [ e^{s(1+\gamma_n)
\inf_{v \in T_{n - k}} V(v)}] \1_{F_1} \big]\, ,\eal
\]
where in the last step we took expectation conditionally on the weights on the spine,
and we twice used that the weights that are not on the spine are independent and
their distribution is not affected by the change of measure, see Proposition \ref{spine_proposition}.
Now, by~\cite[Prop. 5.1]{HS09},
there is a $s_0 \in (0,1)$ such that for all $s\leq s_0$ and any $\eps > 0$,
there exists $C = C(s) > 0$ such that
\[ \E\big[\exp \{ s  \inf_{|v| = n} V(v) \} \big] \leq C n^{(3+\eps)s / 2} \, .\]
In fact, the result in~\cite{HS09} is stated for some fixed $s_0$, but
by H\"older's inequality it immediately translates to all smaller $s \leq s_0$.
Substituting back into the above display (and noting that we can
absorb the $(1+n^{-\delta})$ in front of the infimum into the $\eps$),
we obtain
\be{eqn:event_B_rough_bound}
\bal E_\Q\big[ \overline Y_n^{\, -s}  & (g(n) \vee V^+(\xi_n)^\alpha) e^{-\gamma_n V(\xi_n)}\1_{F_1}\big]\\
& \leq C n^{(3+\eps)s / 2}\E_\Q \big[ (g(n) \vee V^+(\xi_n)^\alpha) g(n)^{-s} e^{s(1+\gamma_n) V(\underline \xi_n) - \gamma_n V(\xi_n) } \1_{F_1} \big] . \eal
\ee
From \eqref{eqn:event_B_rough_bound}, in the case that $\gamma_n \geq 0$, we use that
$V(\xi_n) \geq \underline V(\xi_n)$
and that $\underline V(\xi_n) < - \kappa \log n$ on the event $F_1$ to obtain
\[ \bal \E_\Q\big[ &\overline Y_n^{\, -s}  (g(n) \vee V^+(\xi_n)^\alpha) e^{-\gamma_n V(\xi_n)}\1_{F_1} \big]  \\
& \leq C n^{(3+\eps)s/2} g(n)^{-s} n^{-\kappa (s(1+\gamma_n) + \gamma_n)} \E_\Q [ (g(n) \vee V^+(\xi_n)^{\alpha}) \1_{F_1}]. \eal \]
Since $V(\xi_n)$ has the same distribution as $S_n$, standard random walk computations show that the latter expectation is less than $C(g(n) \vee n^{\alpha/2})$. Consequently, by the choice of $\alpha$, the right hand side of the last inequality is $o(1)$,
provided we choose $\kappa \geq \kappa_0$, where $\kappa_0$ has to be chosen
large enough.

In the case that $\gamma_n \leq 0$, we obtain an upper bound on~(\ref{eqn:event_B_rough_bound}) of
\begin{align}\label{eqn:1601}
 C g(n)^{-s} n^{(3+\eps)s}n^{-\kappa s + o(1)} \E_\Q [ (g(n) \vee V^+(\xi_n)^\alpha) e^{- \gamma_n V(\xi_n)} \1_{F_1}] \,.
\end{align}
We can further bound the expectation using Cauchy-Schwarz to obtain
\begin{align*}
 \bal \E_\Q [ (g(n) \vee V^{+}(\xi_n)^{\alpha}) e^{- \gamma_n V(\xi_n)} \1_{F_1}] & \leq \E_\Q [ (g(n) \vee V^+(\xi_n)^{\alpha})^2]^{1/2}
 \E_\Q [ e^{-2\gamma_n V(\xi_n)}]^{1/2}\\ & \leq (g(n) \vee n^{\alpha/2})\E_\Q [ e^{-2\gamma_n V(\xi_n)}]^{1/2} . \eal
\end{align*}
Since we are considering $\gamma_n \leq 0$ we have that $\delta \geq 1/2$, and since the $V(\xi_n)$ has the distribution of $S_n$ (which is a mean zero random walk with exponential moments), it follows that the expectation in the latter expression is of constant order. Hence \eqref{eqn:1601} is of order $o(1)$, again if $\kappa \geq \kappa_0$
for some suitably chosen $\kappa_0$.

{\bf Step 2.} We now show that
\begin{align*}
\E [\overline{Y}_n^{1-s} \1_{\overline{Y}_n \geq 1} \1_{F_2} ] = \E_{\Q}[ \overline{Y}_n^{-s} (g(n) \vee V^+(\xi_n)^{\alpha}) e^{-\gamma_n V(\xi_n)} \1_{\overline{Y}_n \geq 1} \1_{F_2} ] \to 0
\end{align*}
as $n \to \infty$. We upper bound the latter expression by
\begin{align*}
  \E_{\Q}\big[ \overline Y_n^{\, -s} (g(n) \vee V^+(\xi_n)^\alpha) e^{-\gamma_n V(\xi_n)} \1_{ \overline Y_n \geq 1}\1_{F_2}\big] &\leq \E_{\Q} \big[ (g(n) \vee V^+(\xi_n)^\alpha) e^{-\gamma_n V(\xi_n)} \1_{F_2}\big] \\
 &\leq g(n) e^{ \kappa n^{-\delta} \log n} \Q (F_2) \, ,
\end{align*}
where we used that if $\gamma_n \geq 0$, then we can bound
$V(\xi_n) \geq V(\underline \xi_n) \geq - \kappa \log n$, and
if $\gamma_n \leq 0$, then $- \gamma_n V(\xi_n) \leq 0$.
However, using that $V(\xi_i)$ is a mean zero random walk we may upper bound $\Q(F_2)$ by $Cn^{-3/2} (\log n)^3$ for some $C = C(\kappa)$, which corresponds
to the probability that a random walk comes back to zero at time $n$
on the event that it stays positive, see e.g.~\cite[Lemma A.1]{AidekonShi_weak}. Thus
\begin{align*}
g(n) \Q(F_2) \leq C n^{\frac32 - 2 \delta} n^{-\frac32} (\log n)^3 = o(1),
\end{align*}
which completes the proof.
\end{proof}

\section{Lower Bounds \label{section:lower_bound}}

The goal of this section to find a lower bound on the fractional
moment $\E [ (W_n^{\pm, \delta})^{1-s}]$ in terms of an expression
that only involves a (non-trivial) functional of a random walk.
By Remark~\ref{rem:large_neg_pert} and Section \ref{sec:large_negative_perturbation} we already have the required bounds for negative perturbations, hence it suffices to consider
only positive perturbations.

Let $g(n)$ be as in \eqref{eqn:growth_func} and let $\gamma_n = n^{-\delta}$. We lower bound $g(n) W_n^{+,\delta}$ by
\[ g(n)W_n^{+,\delta} \geq \sum_{|v|=n} (g(n) \wedge V^+(v)^\alpha)e^{-(1+ n^{-\delta}) V(v)} =: \underline Y_n , \]
where $\alpha = \alpha(\delta)$ is chosen as
\[ \alpha(\delta) = \left\{ \ba{ll} 1 &\mbox{if } \delta \geq \frac{1}{2}, \\
   \frac{3}{2\delta} - 2 & \mbox{if } \delta \in (0,\frac{1}{2}). \ea \right. \]

\begin{prop}\label{prop:low_bound} For any $s \in (0,1)$, there exist
 constants $\kappa^* = \kappa^*(s)$ and $\gamma(s)$ such that for $n_0 = \lceil \kappa^*(\log n)^2 \rceil$,
 \[ \E[ \underline Y_n^{1-s} ] \geq  \frac{1}{n_0^{\gamma(s)}} \E_\Q[(g(n-n_0)\wedge (S_{n-n_0}^+)^\alpha)e^{-(n-n_0)^{-\delta} S_{n-n_0}} \1_{\{ \min_{j\leq n-n_0} S_j \geq 0 \}} ]. \]
\end{prop}

The proof of the proposition splits into two lemmas. We first estimate the fractional moments of $\underline Y_n$
with an expression that only involves the weights along the spine.

\begin{lemma}\label{le:spine} Let $\kappa > 0$ and define $n_0 = \lceil (\kappa \log n)^2\rceil$. For any $s \in (0,1)$ there exists a constant $c > 0$ and $\gamma(s)>0$ such that
for all $n$ sufficiently large,
\[ \E [ \underline Y_n^{1-s} ]
 \geq c\, \E_\Q \Big[ \frac{ (g(n) \wedge V^+(\xi_n)^\alpha) e^{-n^{-\delta} V(\xi_n)}}
{ (\log n)^{\gamma(s)} \sum_{j=0}^{n_0-1} e^{-s(1+n^{-\delta}) V(\xi_{j})}+  g(n)  \sum_{j=n_0}^n e^{- s V(\xi_j)}} \Big]
\]
\end{lemma}

\begin{proof}
Note that
\begin{align}
\E \left[ \underline Y_n^{1-s} \right] & = \E \left[ (\underline Y_n)^{-s} W_n \E_{\Q} \left[ (g(n) \wedge V^+(\xi_n)^\alpha) e^{-n^{-\delta} V(\xi_n)} | \calW_{\infty} \right] \right]\notag\\
& = \E_{\Q} \left[ (\underline Y_n)^{-s} (g(n) \wedge V^+(\xi_n)^\alpha) e^{-n^{-\delta} V(\xi_n)} \right] \notag \\
& = \E_{\Q} \left[ \E_{\Q} \left[ (\underline{Y}_n)^{-s} | \calG_n \right] (g(n) \wedge V^+(\xi_n)^\alpha) e^{-n^{-\delta} V(\xi_n)} \right].
\label{eqn:fractional_moment}
\end{align}
We first use Jensen's inequality to estimate $\E_{\Q}[\underline{Y}_n^{-s} | \calG_n ] \geq \E_{\Q}[\underline{Y}_n^s | \calG_n]^{-1}$, and then estimate the latter by grouping the terms in $\underline{Y}_n$ according to the generation at which they first deviate from the spine. This gives us the expression	
\begin{align}
\underline Y_n &= \sum_{j=1}^{n} e^{-(1+n^{-\delta}) V(b_j)}  \!\!\!\! \sum_{v \in T_{n-j}(b_j)} \!\!\!\! (g(n) \wedge V^+(v)^\alpha) e^{-(1+n^{-\delta}) (V(v) - V(b_j))} \notag\\
& \hspace{2cm} + (g(n) \wedge V^+(\xi_n)^\alpha) e^{-(1+n^{-\delta}) V(\xi_n)},
\label{eqn:lower_spine_decomposition}\end{align}
where $b_j$ is the sibling of $\xi_j$ in the tree. Call the summands on the right hand side $U_n^j$:
\begin{align}\label{unj_def}
U_n^j :=  e^{-(1+\gamma_n) V(b_j)}  \!\!\!\! \sum_{v \in T_{n-j}(b_j)} \!\!\!\! (g(n) \wedge V^+(v)^\alpha) e^{-(1+n^{-\delta}) (V(v) - V(b_j))}.
\end{align}
Then by the subadditivity inequality $(\sum_i a_i)^s \leq \sum_i a_i^s$ for $a_i \geq 0$, we have
\begin{align*}
\E_{\Q} \left[ \underline{Y}_n^{s} | \calG_n \right] \leq \sum_{j=1}^n \E_{\Q} \left[ (U_n^j)^s | \calG_n \right] + g(n)^s e^{-s(1 + n^{-\delta}) V(\xi_n)}.
\end{align*}
We now proceed to upper bound the expectation terms. First observe that
\begin{align*}
U_n^j &\leq \!\! \sum_{v \in T_{n-j}(b_j)} \!\!\! g(n) \1_{\{V(v) > 0\}} e^{-(1 + n^{-\delta})V(v)} \\
&\leq \!\! \sum_{v \in T_{n-j}(b_j)} \!\!\! g(n) e^{-V(v)} = g(n) e^{-V(\xi_{j-1})} e^{-(V(b_j) - V(\xi_{j-1}))} \!\! \sum_{v \in T_{n-j}(b_j)} \!\! e^{-(V(v) - V(b_j))}
\end{align*}
from which, using Proposition~\ref{spine_proposition}, we get the simple inequality
\begin{align*}
\E_{\Q} \left[ (U_n^j)^s | \mathcal{G}_n \right] &\leq \E_{\Q} \left[ U_n^j | \mathcal{G}_n \right]^s \\
&\leq C g(n)^s e^{-s V(\xi_{j-1})} \E \Big [ \sum_{|v|=n-j} e^{-V(v)} \Big ]^s \leq C g(n) e^{-sV(\xi_{j-1})}.
\end{align*}
The first and last inequalities both use that $s \in (0,1)$. We only use this bound for $j > n_0$. In the case $j \leq n_0$ we replace the minimum in \eqref{unj_def} by $g(n)$ and use parts (iii) and (iv) of Proposition~\ref{spine_proposition} to get the following upper bound:
\begin{align*}
\E_{\Q} \left[ (U_n^j)^s | \mathcal{G}_n \right] \leq C g(n)^s e^{-s(1+n^{-\delta})V(\xi_{j-1})} \E \Big[ \Big( \sum_{|v|=n-j} e^{-(1+n^{-\delta})V(v)} \Big)^s \Big].
\end{align*}
We claim the expectation term is further bounded above as follows:
\begin{align}
\E \Big[ \Big(\sum_{|v| = n-j} & e^{- (1 + n^{-\delta}) V(v) }\Big)^s\Big] \notag\\
& \leq
\E \Big[ e^{s((n-j)^{-\delta} - n^{-\delta}) \sup_{|v| = n-j} V(v) }
 \Big(\sum_{|v| = n-j}  e^{- (1 + (n-j)^{-\delta}) V(v) }\Big)^s\Big] \notag \\
& \leq (1 + o(1)) \E [ \big(W_{n-j}^{+,\delta}\big)^s] . \label{eq:2103-1}
\end{align}
We temporarily delay the proof of \eqref{eq:2103-1}. Assuming it is true and combining it with the previous display we obtain that for $j \leq n_0$,
\[ \E_\Q [(U_n^j)^s | \calW_\infty] \leq C e^{ - s(1+ n^{-\delta}) V(\xi_{j-1})} g(n)^s \E \big[ \big(W_{n-j}^{+, \delta} \big)^s \big] \leq C e^{ - s(1+ n^{-\delta}) V(\xi_{j-1})} \log n. \]
The logarithmic factor in the last inequality is from Propositions \ref{lemma:upper_reduction_to_random_walk} and \ref{prop:random_walk_expression_bounded},
which are proved independently in Sections~\ref{section:upper_bound} and~\ref{section:random_walk_expression}, respectively.
The estimates on $\E_\Q[ (U_n^j)^s | \calG_n]$ (for $j \leq n_0$ and $j > n_0$) combined with
the spine decomposition~(\ref{eqn:lower_spine_decomposition}) yield
the statement of the lemma.

Finally, it remains to prove the claim~(\ref{eq:2103-1}).
First, note that by a standard application of Chebychev,
for any vertex $v$ with $|v| = n$, and any $\ell \geq 0$,
\begin{equation}\label{eq:2103-2} \bal \p \big\{ \sup_{|v|= n} V(v) \geq \ell n \} & \leq 2^n
 \p \{ V(v) \geq \ell n \}  \leq 2^n e^{- \ell n} \E [ e^{V(v)}]
\\ &
= 2^n e^{- \ell n} \E [ e^{V(v_1)}]^n
\leq e^{(\ell_0 - \ell)n },
\eal
\end{equation}
where we define $\ell_0 =  \lceil \log 2 + \log \E[ \exp\{ V(v_1)\}] \rceil$.
For $j \leq n_0 = \lceil (\kappa \log n)^2\rceil$, noting that
$((n-j)^{-\delta} - n^{-\delta})$ is of order
$n^{-1 - \delta} ( \log n)^2$, we therefore obtain
\begin{align}\label{eq:2103-3} \E \Big[ &e^{s((n-j)^{-\delta}  - n^{-\delta}) \sup_{|v| = n-j} V(v) }
 \Big(\sum_{|v| = n-j}  e^{- (1 + (n-j)^{-\delta}) V(v) }\Big)^s\Big] \notag \\
& \leq e^{n^{-\delta + o(1)}  (\ell_0 + 1)}
\E [ (W_{n-j}^{+,\delta})^s \1_{\{ \sup_{|v| = n-j} V(v) \leq (\ell_0 + 1)(n-j)\}}] \\
& \quad\quad +
\E \Big[ e^{pn^{-\delta -1 + o(1)}\sup_{|v| = n-j} V(v) }
\1_{\{\sup_{|v|= n-j} V(v) \geq (\ell_0 + 1)(n-j)\}} ]^{1/p}
\  \E [ (W_{n-j}^{+,\delta})^{sq} ]^{1/q} , \notag
\end{align}
where in the last step we used H\"older's inequality with
conjugates $p,q \geq 1$ such that $sq < 1$. Now, the first summand
on the right hand side is of order $(1+o(1)) \E[ (W_n^{+,\delta})^s]$,
so that it only remains to consider the second term, which
we can bound using~(\ref{eq:2103-2}),
\[ \bal
\E \Big[  &e^{pn^{-\delta -1 + o(1)}\sup_{|v| = n-j} V(v) }
\1_{\{\sup_{|v|= n-j} V(v) \geq (\ell_0 + 1)(n-j)\}} ]^{1/p} \\
& \quad \leq  \sum_{i \geq 1}e^{pn^{-\delta + o(1)} (\ell_0 + i+1) } \p \big\{\sup_{|v|= n-j} V(v) \geq (\ell_0 + i)(n-j)\big\}^{1/p} \\
 & \quad\leq \sum_{i \geq 1} e^{pn^{-\delta + o(1)} (\ell_0 + i+1) } e^{- \frac{1}{p} i (n-j)}
\leq C e^{-n\frac{1}{p}(1+o(1))}.
\eal
\]
Hence, we obtain from~(\ref{eq:2103-3}) that
\[ \bal
 \E \Big[ e^{s((n-j)^{-\delta}  - n^{-\delta}) \sup_{|v| = n-j} V(v) }
 &\Big(\sum_{|v| = n-j}  e^{- (1 + (n-j)^{-\delta}) V(v) }\Big)^s\Big] \\
& \leq (1+o(1)) \E[ (W_n^{+,\delta})^s ]	
+ C \E[ (W_n^{+, \delta})^{s q}]^{1/q} e^{- \frac{1}{p}n(1+o(1))}.
\eal \]
The second term is exponentially small by the upper bound in Theorem \ref{thm:fractional_moments} (which is proved independently in Sections \ref{section:upper_bound} and \ref{section:random_walk_expression}). This proves \eqref{eq:2103-1}.
\end{proof}

In the next lemma, we simplify the lower bound in Lemma~\ref{le:spine}
by substituting in a suitable strategy for the weights on the spine.
Recall that these weights are in distribution equal to the
random walk $(S_n)_{n \geq 0}$, see Section~\ref{sec:spine_technique}.
In particular, the next lemma is simply a statement about functionals
of a random walk.

\begin{lemma}\label{le:low_strat} Let $(S_n)_{n \geq 0}$ be a centered
random walk started at $0$. For any $\gamma > 0$, $s \in (0,1)$,
 $\kappa \geq \kappa_0:= \frac{3}{s}$
and $n_0 = \lceil ( \kappa \log n)^2 \rceil$, there exists a constant
 $c = c(\kappa)>0$ such that
\begin{align}\label{eq:1903-1} \E_\Q \Big[  & \frac{ (g(n) \wedge (S_n^+)^\alpha) e^{-n^{-\delta} S_n}}
{ (\log n)^{\gamma}\sum_{j=0}^{n_0-1} e^{-s(1+n^{-\delta}) S_j} +  g(n) \sum_{j=n_0}^n e^{-s S_j}} \Big] \\
& \quad\quad \geq
\frac{c}{ \kappa^3(\log n)^{3+\gamma} } \E_\Q \big[ ( g(n-n_0) \wedge (S_{n-n_0}^+)^\alpha) e^{-(n-n_0)^{-\delta} S_{n- n_0}}\1_{\{\underline S_{n - n_0}  \geq - \kappa \log (n-n_0) \}} \big] , \notag
  \end{align}
where $\underline{S}_n = \min_{i=1, \ldots, n} S_i$.
\end{lemma}

\begin{proof} We formulate an event which gives a suitable
strategy for the random walk to achieve the lower bound. Namely, define the event
\[ E =
\left\{ \ba{ll}
S_j \geq 0 & \mbox{for all } j = 0, \ldots, n_0 \\
2 \sqrt{n_0} \leq  S_{n_0} \leq n_0 & \\
 S_j \geq \sqrt{n_0} & \mbox{ for all }
 j = n_0 + 1, \ldots, n  ,
\ea \right\}
\]
where we recall that $n_0 = \lceil (\kappa \log n)^2 \rceil$.
Note that on the event $E$,
we can estimate the denominator on the left hand side of~(\ref{eq:1903-1})
as follows: for any $n_0 = \lceil (\kappa \log n)^2 \rceil$ with $\kappa \geq \kappa_0 := \frac{3}{s}$,
\[ \bal  (\log n)^{\gamma} \sum_{j=0}^{n_0-1} e^{-s(1 + n^{-\delta})S_j}& +  g(n) \sum_{j=n_0+1}^n e^{- s S_j}
  \leq n_0(\log n)^{\gamma} +  g(n) \sum_{j= n_0+1}^n e^{- s\sqrt{n_0}} \\
& \leq  \lceil (\kappa\log n)^2  \rceil(\log n)^{\gamma} +  n g(n) e^{ - s\kappa \log n}  \\
& \leq \lceil (\kappa \log n)^2  \rceil (\log n)^{\gamma}+  n^{\frac{5}{2}} n^{- s\kappa_0}
 \leq \kappa^2 (\log n)^{2+\gamma} (1 + o(1)) \eal \]
where we used that $g(n) \leq n^\frac{3}{2}$ for all $\delta > 0$,
and our choice of $\kappa_0 $ ensures that the second
term is of order  $o(1)$.
Thus, by introducing the event $E$, we obtain the following lower bound
\[\bal   \E_\Q \Big[ & \frac{ (g(n) \wedge (S_n^+)^\alpha) e^{-n^{-\delta} S_n}}
{ (\log n)^{\gamma}\sum_{j=0}^{n_0-1} e^{-s(1+n^{-\delta})S_j} +  g(n) \sum_{j=n_0}^n e^{- sS_j}} \Big] \\
 &\hspace{2cm} \geq  \frac{1+o(1)}{\kappa^2 (\log n)^{2+\gamma}} \E_\Q \Big[ \1_E\, (g(n) \wedge (S_n^+)^\alpha) e^{-n^{-\delta} S_n} \Big] . \eal \]
Using first that on the event $E$, $S_{n_0} \leq n_0 = \lceil  (\kappa\log n)^2 \rceil$
and invoking the Markov property at time $n_0$,
the expectation in the above right hand side can be bounded by
\[ \bal \E_\Q \Big[\1_E\,  & (g(n) \wedge (S_n^+)^\alpha) e^{-n^{-\delta} S_n} \Big] \\
 & \geq \E_\Q \Big[ \1_E \,  
 (g(n) \wedge ((S_n-S_{n_0})^+)^\alpha)e^{-n^{-\delta} (S_n- S_{n_0})  -  n^{-\delta}\lceil (\kappa \log n)^2 \rceil}  \Big]\\
& \geq  \Q \{ \underline S_{n_0} \geq 0; \sqrt{2 n_0} \leq   S_{n_0} \leq n_0 \}
\\ & \quad\quad\times
\E_\Q \big[ ( g(n-n_0) \wedge (S_{n-n_0}^+)^\alpha) e^{-(n-n_0)^\delta S_{n- n_0}}\1_{\{\underline S_{n - n_0}  \geq - \kappa \log (n-n_0) \}} \big].
\eal \]
To complete the proof we only need to show that the first term of the last line is bounded below. We have
\[ \bal  \Q \{ \underline S_{n_0} \geq 0 ; \sqrt{2 n_0} \leq   S_{n_0} \leq n_0 \}
& = \Q \{ \sqrt{2 n_0} \leq   S_{n_0} \leq n_0 \, |\, \underline S_{n_0} \geq 0 \}
\Q \{ \underline S_{n_0} \geq 0 \} \\
& =  c\, \Q \{ \sqrt{2 n_0} \leq   S_{n_0} \leq n_0 \, |\, \underline S_{n_0} \geq 0 \}
\, n_0^{-\frac{1}{2}}(1+o(1)),
\eal  \]
where we used a standard random walk computation, see e.g.~\cite[Thm. A]{K76}. Moreover,
$\Q \{ \sqrt{2 n_0} \leq   S_{n_0} \leq n_0 \, |\, \underline S_{n_0} \geq 0 \}$
converges to a constant depending on $\kappa$, since the random walk conditioned to
stay positive converges to the Brownian meander, see e.g.~\cite{Bol76}.
\end{proof}

The proof of Theorem~\ref{prop:low_bound} follows by combining
Lemmas~\ref{le:spine} and~\ref{le:low_strat}.

\section{Evaluating the random walk expression \label{section:random_walk_expression}}

In this section, we evaluate the functionals of a simple random walk,
which we have encountered in the proofs of the upper and lower bounds
respectively. These are sufficiently similar to be treated by the same
techniques.

Recall that
\[ g(n) = \begin{cases} n^\frac{1}{2} & \mbox{if } \gamma_n = \pm n^{-\delta}, \delta \geq \frac{1}{2}, \\
           n^{\frac{3}{2}- 2\delta} & \mbox{if } \gamma_n = n^{-\delta}, \delta \in (0,\frac{1}{2}) .
          \end{cases}
\]
Moreover, $\alpha = \alpha(\delta)\geq 1$ is defined so that $n^{\frac{\alpha}{2}} = g(n)$.
Also, recall that $(S_n)_{n\geq0}$ is a centered random walk
whose increments have all exponential moments. Denote by
$\underline S_n = \min_{1\leq j \leq n} S_j$.

\begin{prop}\label{prop:random_walk_expression_bounded}
Suppose either $\delta  > 0$ and $\gamma_n = n^{-\delta}$ or otherwise $\delta \geq \frac{1}{2}$
and $\gamma_n = - n^{-\delta}$.
Then, there exists a constant $\kappa_0$ depending only on $\alpha$ and
the distribution of $S_1$ such that
for any $\kappa \geq \kappa_0$, there exist constants $c,C$
such that
\[ c \leq \E_\Q \Big[\1_{\{\underline S_n \geq -\kappa \log n; S_n \geq 0\}} (g(n) \star (S_n^+)^\alpha) e^{-\gamma_n S_n} \Big]\leq C\log n , \]
where $\star$ is either $\wedge$ or $\vee$.
\end{prop}

We will prove this proposition in two steps. First, in Lemma~\ref{lemma:KMT_coupling},
we will show that we can replace the functional of a random walk
by an equivalent functional of a Brownian motion. Here, we will use
the coupling of a random walk with a Brownian motion due to
Koml\'os-Major-Tusn\'ady.
Finally, we can
evaluate that expression which is only a functional of the end point
of the Brownian motion and its maximum using the explicit formula of the
their joint density,
see Lemma~\ref{lemma:functional_Brownian_motion}.

In what follows we let $(B_t)_{t \geq 0}$ denote a standard Brownian motion started at the origin and $\E_0$ denote expectation with respect to this Brownian motion.

\begin{lemma}\label{lemma:KMT_coupling} Let $\sigma^2 = \Var(S_1)$.
Under the assumptions of Proposition~\ref{prop:random_walk_expression_bounded},
there exists $\kappa_0 > 0$ such that for any $\kappa \geq \kappa_0$
there exist constants $c,C > 0$ (depending only on $\delta$ and the distribution of
$S_1$) such that
\[ \bal
 c\,\E_0 [ & (g(n) \star  (B_n^+)^\alpha) e^{- \gamma_n \sigma B_n } \1_{\{ \underline B_n \geq - \underline \kappa \log n, B_n \geq 0\}} ] + O(n^{-(1 \wedge 2 \delta) + o(1)}) \\
& \leq    \E_\Q[ (g(n) \star (S_n^+)^\alpha ) e^{-\gamma_n S_n}  \1_{\underline S_n \geq -\kappa \log n, S_n \geq 0} ] \\
& \quad \leq C\,\E_0 [ (g(n) \star  (B_n^+)^\alpha) e^{- \gamma_n \sigma B_n } \1_{\{ \underline B_n \geq - \overline \kappa \log n, B_n \geq 0\}} ] + O(n^{-(1 \wedge 2 \delta) + o(1)})
  ,\eal\]
where $\overline \kappa, \underline \kappa \geq 0$ are some suitable constants
(depending on $\kappa$ and the distribution of $S_1$) and
$\underline B_n = \min_{0 \leq t \leq n} B_t$.
\end{lemma}

\begin{proof}
Let $(S_t)_{t\geq 0}$ denote the piecewise constant approximation of $(S_n)_{n \geq 0}$,
defined by $S_t = S_{\lfloor t \rfloor }$. Since the increments of the random walk have exponential moments,
the Koml\'os-Major-Tusn\'ady theory~\cite{KMT2}
provides a coupling of $(S_t)_{t \in [0,n]}$ and a standard Brownian motion
$(B_t)_{t \in [0,n]}$ such that for any $\rho > 0$, there
exists a constant $c_\rho > 0$ (depending on $\rho$ and the distribution
of $S_1$) satisfying
\[ \p \Big\{ \sup_{s \in [0,n]} | S_s - \sigma B_s| \geq c_\rho \log n \Big\} \leq n^{-\rho} .\]
This is an easy extension of the original result, see e.g.\ the proof
of Thm.~2.6. in~\cite{AD11}.
Denote by $E = \{ \sup_{s \in [0,n]} | S_s - \sigma B_s| \leq c_\rho \log n\}$.
It will be convenient to choose $\rho = 4 \alpha$ and
especially for the lower bound set $\kappa_0 = \frac{c_\rho}{\sigma}$.
From now on we will assume that  $\kappa \geq \kappa_0$.

{\bf Step 1.} \emph{Upper bound on the event $E$}.
On the event $E$ the coupling works well and
we can replace $(S_t)_{t \in [0,n]}$ by $(B_t)_{t \in [0,n]}$ in the following sense
\[ \bal (g(n) & \star (S_n^+)^\alpha)  e^{-\gamma_n S_n} \1_{\{ \underline S_n \geq -\kappa \log n,
 S_n \geq 0 \}} \\
& \leq
(g(n) \star (\sigma B_n^+ + c_\rho \log n)^\alpha) e^{-\gamma_n \sigma B_n + c_\rho |\gamma_n|\log n} \1_{\{ \underline B_n \geq -\overline \kappa \log n,
\sigma B_n \geq - c_\rho \log n \}} ,
\eal
\]
where $\overline \kappa := \frac{\kappa + c_\rho}{\sigma}$.
Now, note that $|\gamma_n| \log n \ra 0$ as $n \ra \infty$, so that we can bound
$e^{c_\rho|\gamma_n|\log n}$ by a constant and further we can bound
the sum $(B_n^+ + c_\rho \log n)^\alpha \leq 2^\alpha( B_n^+ \vee c_\rho \log n)$. Hence, we find that
\be{eqn:upper_bound_on_RW} \bal\E  \big[&(g(n)  \star (S_n^+)^\alpha)  e^{-\gamma_n S_n} \1_{\{ \underline S_n \geq -\kappa \log n,
 S_n \geq 0 \}} \1_E \big]\\
& \leq C 2^\alpha \E \big[(g(n) \star  ((\sigma B_n^+)^\alpha \vee (c_\rho \log n)^\alpha)) e^{-\gamma_n \sigma B_n} \1_{\{ \underline B_n \geq -\overline \kappa \log n,
 \sigma B_n \geq - c_\rho \log n \}} \big]
\eal \ee
Now if $\star = \vee$, then $g(n) \star  ((\sigma B_n^+)^\alpha \vee (c_\rho \log n)^\alpha)
= g(n) \star  (\sigma B_n^+)^\alpha$ and on the other hand if $\star = \wedge$, we have to estimate
\[\bal \E [(g(n) & \wedge  ((\sigma B_n^+)^\alpha \vee (c_\rho \log n)^\alpha)) e^{-\gamma_n B_n} \1_{\{ \underline B_n \geq -\overline \kappa \log n ,\,
\sigma  B_n \geq - c_\rho \log n \}} ] \\
& \leq \E [(g(n)  \wedge  (\sigma B_n^+)^\alpha ) e^{-\gamma_n\sigma B_n} \1_{\{ \underline B_n \geq -\overline \kappa \log n ,\,
\sigma  B_n \geq  c_\rho \log n \}} ] \\
& \quad + \E [  (c_\rho \log n)^\alpha  e^{-\gamma_n\sigma B_n} \1_{\{ \underline B_n \geq - \overline \kappa \log n, - c_\rho \log n\leq \sigma B_n \leq c_\rho \log n  \}} ]
\eal \]
We now claim that the second summand in the previous display is of order $o(n^{-1})$.
Indeed,
\[\bal \E [  (c_\rho \log n)^\alpha & e^{-\gamma_n \sigma B_n} \1_{\{ \underline B_n \geq - \overline \kappa \log n, - c_\rho \log n\leq \sigma B_n \leq c_\rho \log n  \}} ] \\
& \leq
(c_\rho \log n)^\alpha  e^{|\gamma_n| c_\rho \log n} \p \{ \underline B_n \geq -\overline \kappa \log n, - c_\rho \log n\leq \sigma B_n \leq c_\rho \log n  \}
\\ & \leq C n^{-\frac{3}{2} + o(1)} 
\eal \]
where the last bound follows from a standard Brownian calculation using
for example the explicit density of maximum and final position (see
e.g.\ the proof of Lemma~\ref{lemma:functional_Brownian_motion}).

Hence, we can summarize the two possible choices for $\star$ and conclude from~(\ref{eqn:upper_bound_on_RW})
that
\[ \bal \E  \big[&(g(n)  \star (S_n^+)^\alpha)  e^{-\gamma_n S_n} \1_{\{ \underline S_n \geq -\kappa \log n,
 S_n \geq 0 \}} \1_E \big] \\
& \leq C \, \E [(g(n)  \star  (\sigma B_n^+)^\alpha ) e^{-\gamma_n \sigma B_n} \1_{\{ \underline B_n \geq -\overline \kappa \log n, \,  \sigma B_n \geq - c_\rho \log n \}} ] + O(n^{-\frac{3}{2} + o(1)}).
\eal
\]
This is almost of the right form for the main term in the statement of the lemma
(where the $\sigma$ in front of $B_n^+$ can be absorbed into the constants).
Thus it remains to show that we can replace the indicator $\sigma B_n \geq - c_\rho \log n$ by that of $B_n \geq 0$ to obtain the correct upper bound on the event $E$.

Here, it suffices to show that the following expression is of order $O(n^{-2(\frac{1}{2} \wedge \delta) + o(1) })$,
\[ \bal \E [(g(n)  & \star  (\sigma B_n^+)^\alpha ) e^{-\gamma_n \sigma B_n} \1_{\{ \underline B_n \geq - \overline \log n, \,  - c_\rho \log n \leq \sigma B_n \leq 0  \}} ] \\
& \leq \E [ g(n)   e^{|\gamma_n| c_\rho \log n } \1_{\{ \underline B_n \geq - \overline \kappa \log n, \,  - c_\rho \log n \leq \sigma B_n \leq 0  \}} ] \\
& \leq C g(n) \p \{ \underline B_n \geq -\overline \kappa \log n, \,  - c_\rho \log n \leq \sigma B_n \leq 0  \}
\\
& \leq C g(n) n^{-\frac{3}{2} + o(1)} ,
\eal
\]
where the last step follows from a standard Brownian calculation.
However, if $|\delta| \geq \frac{1}{2}$, then $g(n) = n^\frac{1}{2}$, so that the latter expression
is of order $n^{- 1 + o(1)}$, whereas if $\delta \in (0,\frac{1}{2})$, then $g(n) = n^{\frac{3}{2} - 2\delta}$,
so that the expression is of order $n^{- 2 \delta + o(1)}$ as claimed. This last
step completes the proof of the upper bound on the event $E$.

{\bf Step 2.} \emph{Upper bound on the event $E^c$.}
In this scenario, we can estimate using Cauchy-Schwarz
\[ \bal\E  \big[(g(n)  \star (S_n^+)^\alpha) & e^{-\gamma_n S_n} \1_{\{ \underline S_n \geq -\kappa \log n,
 S_n \geq 0 \}} \1_{E^c} \big]\\
& \leq \E  \big[(g(n)  \vee (S_n^+)^\alpha)^2
  e^{2n^{-\frac{1}{2}} S_n}]^{\frac{1}{2}}
\p(E^c)^{\frac{1}{2}} , \\
\eal 	
\]
where we also used in the last step that if $\gamma_n = - n^\delta$, we only consider the case $\delta \geq \frac{1}{2}$
so that $- \gamma_n S_n \leq n^{-\frac{1}{2}} S_n$, while if $\gamma_n \geq 0$ this bound holds trivially since $S_n \geq 0$.  Using that $g(n) \leq n^{\frac{\alpha}{2}}$, we have that this
expression can be bounded from above by
\be{eqn:estimate_on_E_c} n^{\alpha}\ \E [ \big[(1  \vee (n^{-\frac{1}{2}}S_n^+)^\alpha)^2
  e^{2n^{-\frac{1}{2}} S_n}]^{\frac{1}{2}}
\p(E^c)^{\frac{1}{2}}  \ee
Now, we combine the weak convergence
of $n^{-\frac{1}{2}}S_n$ to $\sigma B_1$ with a standard uniform integrability bound
that follows easily from $S_1$ having exponential moments
to
deduce that
\[ \E [ \big[(1  \vee (n^{-\frac{1}{2}}S_n^+)^\alpha)^2
  e^{2n^{-\frac{1}{2}} S_n}] \ra \E \Big[  ( 1 \vee (\sigma B_1^+)^\alpha)^2
e^{2 \sigma B_1} \Big] . \]
Hence, if we combine this observation with the estimate
$\p(E^c) \leq n^{-\rho}$ we have the following bound
on~(\ref{eqn:estimate_on_E_c}):
\[ n^\alpha \,  \E [ \big[(1  \vee (n^{-\frac{1}{2}}S_n^+)^\alpha)^2
  e^{2n^{-\frac{1}{2}} S_n}]^{\frac{1}{2}}
\p(E^c)^{\frac{1}{2}}
\leq C n^{\alpha- \frac{1}{2} \rho}.  \]
Since we chose $\rho = 4\alpha$ the latter expression is of order
$n^{-\alpha} \leq n^{-1}$ (since $\alpha \geq 1$) as claimed.

A \emph{lower bound} simply follows by interchanging the roles of
random walk and Brownian motion and replacing standard
Brownian calculations by standard random walk calculations, see e.g.~\cite[Lemma A.1]{AidekonShi_weak}. Moreover, we then need to replace the role of $\overline \kappa$
by $\kappa$ and that of $\kappa$ by a suitable $\underline \kappa$. In particular,
we will choose $\underline \kappa := \sigma \kappa - c_\rho$, which is non-negative
if $\kappa \geq \kappa_0 := c_\rho / \sigma$.
\end{proof}

\begin{lemma}\label{lemma:functional_Brownian_motion} Under the assumptions of
Proposition~\ref{prop:random_walk_expression_bounded}, for any $\kappa \geq 0$,
and all $n$ sufficiently large,
\[ c\leq \E_0 \Big[ (g(n) \star  B_n^\alpha) e^{- \gamma_n \sigma B_n } \1_{\{ \underline B_n\geq - \kappa \log n, B_n \geq 0\}}\Big] \leq C \log n . \]
 \end{lemma}

\begin{proof}
We use the explicit formula for the joint density of $B_t$ and its running maximum
$\overline B_t = \sup_{0 \leq s \leq t} B_s$, see e.g.\!~\cite[Thm. 3.7.3]{Shreve04},
which states that
$(B_t, \overline B_t)$ has for fixed $t > 0$ a joint density with respect to
$2$-dimensional Lebesgue measure given by
\[ f(x,m) = \frac{2(2m - x)}{t\sqrt{2\pi t}} e^{-\frac{(2m-x)^2}{2t}}  ,\quad \mbox{for }
x \leq m, m > 0. \]
Thus, we can explicitly calculate the functional of the Brownian motion
and its minimum
by first reflecting the Brownian motion as
\be{all_in_BM} \bal \E_0 \Big[ (g(t) & \star  B_t^\alpha) e^{- \gamma_t \sigma B_t } \1_{\{ \underline B_t \geq - \kappa \log t, B_t \geq 0\}}\Big] \\
& = \E_0 [ (g(t) \star (-B_t)^\alpha)  e^{ \gamma_t \sigma B_t } \1_{\{ \overline B_t \leq  \kappa \log t, B_t \leq 0\}} ] \\
& = \frac{2}{\sqrt{2\pi}} t^{-\frac{3}{2}} \int_0^{\kappa\log t}   \int_{-\infty}^0 (g(t) \star (-x)^\alpha) e^{\gamma_t \sigma x} (2m- x)
e^{-\frac{(2m-x)^2}{2t}} dx \,dm\, .\eal\ee

We will show lower and upper bound separately and also distinguish
the case of a large or a small perturbation.

{\it Upper bound in the case $|\gamma_t| \leq t^\frac{1}{2}$}.
In this case, $g(t) = t^{\frac{1}{2}}$ and $\alpha = 1$,
so we can bound
the expression in~(\ref{all_in_BM}) by
\[ \bal C t^{-\frac{3}{2}}(\kappa\log t) \int_0^\infty &(t^\frac{1}{2} \vee x) e^{t^\frac{1}{2} \sigma x} (2\kappa\log t + x) e^{-\frac{x^2}{2t}} d x \\
& \leq C (\log t) \big(1+ \tfrac{2\kappa\log t}{t^\frac{1}{2}}\big)\int_0^\infty (1\vee x)^2 e^{\sigma x - \frac{x^2}{2}} dx \, ,\eal \]
which is bounded by $C \log t$ for some (different to above) constant $C$.

{\it Upper bound in the case $\gamma_t = t^{-\delta}$ for $\delta \in (0,\frac{1}{2})$}.
Note that here we have defined $g(t) = t^{\frac{3}{2} - 2\delta}$
and that $\alpha$ is chosen so that $t^{\delta\alpha} = g(t)$,
therefore we can bound~(\ref{all_in_BM}) by
\[ \bal C t^{-\frac{3}{2}}(\kappa\log t) & \int_0^\infty (g(t) \vee x^\alpha) (2\kappa\log t + x) e^{-t^{-\delta} x} d x \\
& \leq C t^{-\frac{3}{2}+ 2\delta} g(t)  (\log t)\big(1+ \tfrac{2\kappa\log t}{t^\delta}\big) \int_0^\infty (1 \vee x^\alpha )^2  e^{-x} d x
\,,\eal \]
so that by our choice of $g(t)$, the latter is bounded by $C \log t$.

{\it Lower bound in the case $|\gamma_t|\leq t^{\frac{1}{2}}$}.
Here, we have chosen $g(t) = t^\frac{1}{2}$ and $\alpha = 1$. We can lower
bound the expression in~(\ref{all_in_BM}) by
\[\bal c t^{-\frac{3}{2}} \int_0^\infty (g(t) \wedge x) x & e^{-\sigma t^{-\frac{1}{2}}x}
e^{- \frac{1}{2} t^{-\frac{1}{2}}(2\kappa \log t
 + x)^2} dx \\
& \geq c \int_0^\infty (1 \wedge x) x e^{-\sigma x} e^{- 2\kappa t^{-\frac{1}{2}} \log t  - x^2} dx ,
\eal \]
where we used the inequality $(x+y)^2 \leq 2 (x^2 + y^2)$. This
is expression is bounded from below by an absolute constant.

{\it Lower bound in the case $\gamma_t = t^{-\delta}, \delta \in (0,\frac{1}{2})$}.
Here, we have defined $g(t) = t^{\frac{3}{2} - 2\delta}$ and $\alpha$ is chosen
so that $t^{\delta\alpha} = g(t)$. Then, we
can similarly to above find a lower bound on the integral in~(\ref{all_in_BM})
\[\bal c t^{-\frac{3}{2}} \int_0^\infty (g(t) \wedge x^\alpha) x & e^{-\sigma t^{-\delta}x}
e^{- \frac{1}{2} t^{-\frac{1}{2}}(2\kappa \log t
 + x)^2} dx  \\
& \geq
c \int_0^\infty (1 \wedge x^\alpha) x  e^{-\sigma x}
e^{-  t^{\delta -\frac{1}{2}} (2\kappa \log t)^2
 - t^{\delta - \frac{1}{2}} x} dx ,
\eal \]
which, by dominated convergence, is bounded below by an absolute constant.
\end{proof}

\begin{proof}[Proof of Proposition~\ref{prop:random_walk_expression_bounded}]
 The proof now follows by combining the previous two Lemmas~\ref{lemma:KMT_coupling}
and~\ref{lemma:functional_Brownian_motion}.
\end{proof}

\appendix
\section{Fractional Moment Bounds to Asymptotics \label{sec:moments2asymp}}
	
In this appendix we show how the fractional moment bounds obtained in Theorem~\ref{thm:fractional_moments} imply the asymptotics in the main Theorem~\ref{thm:main}. The arguments are fairly standard and in a variation are also used in~\cite{HS09}.

\begin{lemma}[Upper bounds] Write $W_n^{\cdot, \delta} = W_n^{\pm, \delta}$. Let $(a_n)_{n \in \N}$ be a sequence of real numbers such that $|a_n| \to \infty$ as $n \to \infty$, and suppose that for every $\gamma \in (0,1)$ we have
\begin{align*}
\E[(W_n^{\cdot, \delta})^{\gamma}] = e^{\gamma a_n(1+ o(1))}.
\end{align*}
Then $W_n^{\cdot, \delta} \leq e^{a_n (1 + o(1))}$ in probability, as $n \to \infty$.
Moreover, if for any $\eps > 0$, $\sum_{n \geq 1} e^{- \eps |a_n|} < \infty$, then
$W_n^{\cdot, \delta} \leq e^{a_n (1 + o(1))}$ almost surely.
\end{lemma}

\begin{rem} This lemma shows how to deduce the upper bounds in Theorem~\ref{thm:main}
 from Theorem~\ref{thm:fractional_moments}.
We take
\[ a_n = \begin{cases}  - \frac{1}{2}\log n & \mbox{ for $\delta \geq 1/2$,} \\
                     \frac{1}{2} n^{1-2\delta} \beta_c^2 \la''(\beta_c) & \mbox{ for $W_n^{\cdot, \delta} = W_n^{-,\delta}, 0 < \delta < 1/2$,} \\
                     (2 \delta - \frac{3}{2}) \log n & \mbox{ for $W_n^{\cdot, \delta} = W_n^{+,\delta}, 0 < \delta < 1/2$.}
         \end{cases}
\]
In particular, the lemma shows that in the case of large, negative perturbations
the upper bound holds almost surely.
\end{rem}

\begin{proof}
Fix $\epsilon > 0$ and let $\gamma \in (0,1)$. Then by assumption $\E[(W_n^{\cdot, \delta})^{\gamma}] \leq e^{\gamma a_n + \epsilon \gamma |a_n| /2}$ for all $n$ sufficiently large. Then by Chebyshev's inequality
\begin{align*}
\p(W_n^{\cdot, \delta} > e^{a_n + \epsilon |a_n|}) &\leq e^{-\gamma a_n - \epsilon \gamma |a_n|} \E[(W_n^{\cdot, \delta})^{\gamma}] \\
& \leq e^{-\epsilon \gamma |a_n|/2}
\end{align*}
Thus by the assumption on the $|a_n|$ we have $W_n^{\cdot, \delta} \leq e^{a_n(1 + o(1))}$ in probability. The second part of the statement follows from Borel-Cantelli.
\end{proof}

\begin{lemma}[Lower bounds] Let $W_n^{\cdot, \delta} = W_n^{\pm, \delta}$. Let $(a_n)_{n \in \N}$ be a sequence with $(\log n)^\frac{1}{2} \ll |a_n| \ll n$. Assume that for all $\gamma \in (0,1)$ we have
 \[ \E[(W_n^{\cdot, \delta})^\gamma ]  = e^{\gamma a_n(1 + o(1))} .\]
Then almost surely
\[ W_n^{\cdot, \delta}  \geq e^{a_n(1+o(1))} . \]
\end{lemma}

The lower bounds of Theorem \ref{thm:main} are therefore derived from this lemma and Theorem \ref{thm:fractional_moments} using the same sequence $a_n$ as in the last remark. Note, however, that in this case the lower bounds are almost sure rather than in probability.

\begin{proof} Let $\eps > 0$. By assumption, for any $\gamma \in (0,\frac{1}{2})$ we
have that $\E[ (W_n^{\cdot, \delta})^\gamma] \geq e^{\gamma a_n - \frac{\eps}{4}\gamma |a_n|}$
and $\E[ (W_n^{\cdot, \delta})^{2\gamma} ] \leq e^{2\gamma a_n + \frac{\eps}{4}|a_n|}$, for all $n$ sufficiently large.
By the Paley-Zygmund inequality, we have that
\begin{align}
\p \left(  W_n^{\cdot, \delta} > e^{a_n-\eps|a_n|}  \right) &\geq \Big( 1 - \frac{e^{\gamma ( a_n- \eps|a_n|)}}{\E[ (W_n^{\cdot, \delta})^\gamma]} \Big)^2 \frac{\E[(W_n^{\cdot, \delta})^\gamma]^2}{\E[ (W_n^{\cdot, \delta})^{2\gamma}]} \notag \\
 &\geq \big( 1 -  e^{-3\eps\gamma|a_n|/4} \big)^2 e^{-\frac{3\eps}{4}|a_n|} \geq e^{-\eps|a_n|} \, , \label{rough_bound}
\end{align}
for all $n$ sufficiently large. Now
define $\tau_n = \lceil\frac{2\eps|a_n|}{\log 2}\rceil$
so that $\tau_n < n$ for all $n$ sufficiently large. Then
\begin{align*}
W_n^{\pm, \delta} &= \sum_{|w| = \tau_n} e^{-(1 \pm n^{-\delta})V(w)} \!\!\!\! \sum_{\substack{v \in T(w) \\ |v| = n - \tau_n}} e^{-(1\pm n^{-\delta})(V(v) - V(w))} \\
&\geq \exp \big\{ -(1 \pm n^{-\delta}) \max_{|w| = \tau_n} V(w) \big \} \sum_{|w|=\tau_n} \!\! \sum_{\substack{v \in T(w) \\ |v| = n - \tau_n}} e^{-(1 \pm n^{-\delta})(V(v) - V(w))}.
\end{align*}
Call the rightmost sum $Y_{n-\tau_n}(w)$. Then the above implies the estimate
\begin{align*}
\bal \p \Big( W_n^{\pm, \delta} \leq e^{a_n - \eps |a_n|} \exp\{ -(1 \pm n^{-\delta}) \max_{|w|=\tau_n} V(w) \} \Big) & \leq \p \Big( \sum_{|w| = \tau_n} Y_{n-\tau_n}(w) \leq e^{a_n - \eps|a_n|} \Big) \\
 & \leq \p \left( Y_{n-\tau_n} \leq e^{a_n - \eps|a_n|}  \right)^{2^{\tau_n}} . \eal
\end{align*}
Equation \eqref{rough_bound} shows that this expression is bounded by $\exp \{ - e^{ \eps|a_n|}\}$.
Therefore, by the assumption that $|a_n| \gg (\log n)^{1/2}$ the probabilities
are summable and so by Borel-Cantelli we have that with probability one
\begin{align}\label{eqn:exp_lower_bound}
Y_n \geq e^{a_n - \eps|a_n|} \exp \big \{ -(1 \pm n^{-\delta}) \max_{|w|=\tau_n} V(w) \big	\} \,,
\end{align}
for $n$ sufficiently large. However it is well known  that there is an explicit constant $C > 0$ such that
$\frac{1}{\tau_n } \max_{|v| =\tau_n } V(v) \ra C$
with probability one (the max is the position of the rightmost particle in the system of branching random walks), so that
$$\exp \big \{ -(1 \pm n^{-\delta}) \max_{|w|=\tau_n} V(w) \big	\} = e^{C^*\eps |a_n| (1 + o(1))}$$ for some $C^* > 0$ (not depending on $\eps$). Hence with probability one $Y_n \geq e^{a_n(1+o(1))}$.
\end{proof}

\bibliographystyle{alpha}
\bibliography{critical_window}

\end{document}